\documentclass[12pt]{article}
\usepackage{amsmath,amssymb,amsthm,latexsym,euscript,amscd}
\usepackage[english]{babel}
\usepackage[all]{xy}

\newtheorem{Theorem}{Theorem}
\newtheorem{Proposition}[Theorem]{Proposition}
\newtheorem{Corollary}[Theorem]{Corollary}
\newtheorem{Lemma}[Theorem]{Lemma}

\newtheorem{Conjecture}[Theorem]{Conjecture}

\newcommand{\ZZ}{{\mathbb Z}}
\newcommand{\CC}{{\mathbb C}}
\newcommand{\RR}{{\mathbb R}}

\newcommand{\kk}{{\mathbb K}}

\newcommand{\pg}{{\cal P}(\Gamma)}

\begin{document}

\title{Theory of homotopes with applications to mutually unbiased bases, harmonic analysis on graphs and perverse sheaves}
\author{Alexey Bondal, Ilya Zhdanovskiy\\
}
\date{}

\maketitle

{\bf Abstract.}
    The paper is a survey of modern results and applications of the theory of homotopes. The notion of a well-tempered element in an associative algebra is introduced and it is proven that the category of representations of the homotope constructed by a well-tempered element is the heart of a suitably glued t-structure. Hochschild and global dimensions of the homotopes are calculated in the case of well-tempered elements. The homotopes constructed by generalized Lapalce operators in Poincare groupoids of graphs are studied. It is shown that they are quotients of Temperley-Lieb algebras of general graphs. The perverse sheaves on a punctured disc and on a 2 dimensional sphere with a double point are identified with representations of suitable homotopes. Relations of the theory to orthogonal decompositions of sl(n, C) into the sum of Cartan subalgebras, to classifications of configurations of lines, to mutually unbiased bases, to quantum protocols, to generalized Hadamard matrices are discussed.



\tableofcontents

\section{Introduction}
\footnote{The reported study was funded by RFBR, project number 19-11-50213. The study has been funded within the framework of the HSE University Basic Research Program. The authors were partially supported by the RFBR grant project 18-01-00908. The first author is partially supported by Laboratory of Mirror Symmetry NRU HSE, RF Government grant, ag. № 14.641.31.0001. The work was partially supported by JSPS KAKENHI № JP20H01794. This work was supported by World Premier International Research Center Initiative (WPI Initiative), MEXT, Japan.}
The central object of study for this paper is the so-called homotope. Recall what is this. Assume we are given a unital algebra $A$ with a fixed element $\Delta$. The homotope is a simple algebraic construction that allows us to produce a new algebra from $A$ and $\Delta$. The new algebra $B^+$ coincides with $A$ as a vector space, while the multiplication is modified via $\Delta$:
$$
a\cdot b=a\Delta b
$$
Since the new algebra might not have a unit, we adjoin the unit to it. The resulted algebra we denote by $B^+$, this is the (unital) homotope. Alegbra $B^+$ is the augmentation ideal in $B$.

Despite of the simple definition, the theory of representations of homotopes has deep connactions with with advanced modern constructions in homological algebra. Moreover,  homotopes give a unifying approach to various problems in mathematics and physics. Among those are the classification of configuration of lines with given angles in Hilbert spaces, the description of orthogonal decompositions of simple Lie algebras into the sum of Cartan subalgebras, the discrete harmonic analysis on graphs, the classification of mutually unbiased bases, which have wide application in Quantum Mechanics and Quantum Information Theory, description of perverse sheaves on stratified topological spaces, the representation theory of quasi-heredetary algebras, etc. 

The goal of this paper is to explain how homotopes naturally show up in some of the above-mentioned areas of research, to develop the categorical homological approach to the representation theory of homotopes and to demonstrate applications of homotopes in new situations, such as the algebraic description of perverse sheaves on certain stratified topological spaces.

Original algebra $A$ and its homotope $B$ are related by a pair of homomorphisms $B\to A$. These homomorphisms are essential for understanding the representation theory of homotopes. 

In the abstract framework, we formalize conditions on the element $\Delta\in A$ which imply good categorical description of representation theory of the homotope $B$ in terms of representations of $A$. These are two conditions: one of them requires the two-sided ideal generated by $\Delta$ is the whole algebra $A$, and the other one is that the augmentation ideal $B^+$ is projective as a left and right $B$-module. We call  the elements satisfying these conditions {\em well-tempered}. 
We show that well-temperedness is the property of the double cosets with respect to the action of the group of invertible elements in $A$. We the homological properties of the homotope and  obtain an upper bound for Hochschild and global dimension for the algebra with modified multiplication in terms of the same dimensions for $A$. The condition of the augmented algebra to be projective can be analysed in terms of the Noncommutative Differential Geometry. 

The homotope occurs to be Morita equivalent under a suitably defined homotopy equivalence of elements $\Delta$. This motivates us to keep the historical name 'homotope', which originally appeared in Algebra research about 60 years ago. 

It is interesting that when $A$ is a commutative algebra the homotope construction has the geometric meaning of a divisorial contraction. Though  the element $\Delta$ in this situation is well-tempered only if it is invertible. 

We introduce a class of algebras $B(\Gamma )$ related to a simply laced graphs $\Gamma$ and study
their representation theory. In order to describe the algebra, let us number the edges of the graph by pairs of its vertices. 

Algebra $B(\Gamma )$ depends on a system of invertible parameters $s_{ij}$ which are assigned to the edges $(ij)$ of the graph $\Gamma$. Algebra $B(\Gamma )$ is generated by idempotents $x_i$, numbered by the vertices of the graph, which subject to relations:
\begin{itemize}
\item{$x^2_i = x_i$, for every vertex $i$ of $\Gamma$,}

\item{$x_ix_jx_i = s_{ij}^2x_i$, $x_jx_ix_j = s_{ij}^2x_j$, if $(ij)$ is an edge 
in $\Gamma$,}

\item{$x_ix_j = x_jx_i = 0$, if $(ij)$ is not an edge in $\Gamma$.}

\end{itemize}

Our original motivation to consider algebras $B(\Gamma )$ was the problem of classifying systems of $m$ Cartan subalgebras in Lie algebra $sl(n, \CC )$, pairwise orthogonal with respect to Killing form. This can easily be interpreted in terms of the representation theory of algebras $B(\Gamma_m(n))$, where $\Gamma_m(n)$ is a graph with vertices arranged in $m$ rows, by $n$ vertices in each row, such that vertices in different rows are connected by edges and those in the same row are not. The problem on orthogonal Cartan subalgebras can be reformulated as the problem on finding systems of projectors satisfying the conditions of unbiasedness (see \ref{aup}). These conditions define a representation of $B(\Gamma_m(n))$ for parameters $s_{ij}=\frac 1n$.

 Many other classical and modern
problems in Algebra and Algebraic Geometry, such as Poncelet porism \cite{Sch} and its generalizations, admit interpretations in terms of the representation theory of algebras $B(\Gamma )$ for other graphs.

A unitary version for representation theory for $B(\Gamma )$ is also of big interest. In this case, we consider representations in a Hermitian (resp. Euclidean) space and require that the generating idempotents in $B(\Gamma )$ are represented by Hermitian (resp. orthogonal) projectors. Classification of such representations includes as a partial case the problem of classifying systems of lines in a Hermitian vector space with prescribed angles between them. The above mentioned algebraic version of the problem is interpreted as the  complexification of the unitary problem.

The Hermitian representations for $B(\Gamma_m(n))$ define so-called {\em mutually unbiased bases}. The problem of classifying these bases has recently attracted a lot of attention in Quantum Information Theory due to its relevance to quantum encoding, decoding and quantum tomography.  In \ref{protocols}, we shortly described an instance of how mutually unbiased bases appear in this context on the example of protocols of quantum channels.

Algebra $B(\Gamma )$ is a quotient of Temperley-Lieb algebra and of Hecke algebra of the graph $\Gamma$. 
Thus, the study of its representation theory can be viewed as the first step in studying representation theory of Hecke algebras of complicated graphs. While the rich representation theory of Hecke algebras of Dynkin and extended Dynkin graphs is a part of common knowledge, very little is known about representations of Hecke algebras for graphs with noncommutative fundamental group.

The representation theory for $B(\Gamma )$ is closely related to the harmonic analysis of local systems on the graph.
Here, the crucial fact is that the algebra $B(\Gamma )$ can be obtained via the homotope consruction applied to Poincare groupoid of the graph 
and to $\Delta$ a generalized Laplacian of the graph (see \ref{changemult}, \ref{sectionlaplace}). Moreover, generalized Laplacians in $B(\Gamma )$ are examples of well-tempered elements in Poincare groupoid.
This construction implies a pair of homomorphism of $B(\Gamma )$ into the algebra of Poincare groupoid of $\Gamma$. The representations of $B(\Gamma )$ can be described in terms of representations of Poincare groupoid and functors that are naturally assigned to these homomorphisms.
The representations of Poincare groupoid are identified with local systems on the graph, whence the relevance of the harmonic analysis of the local systems. Also let us note that the Hermitian version of the problem is intimately related to the positivity of Laplace operator.


Poincare groupoid is isomorphic to the matrix algebra over the group algebras of the fundamental group of the graph. In particular, for simply connected graph (in other words, for a tree) it is the matrix algebra over the base field. All nonzero elements in this algebra are well-tempered. It would be interesting to classify well-tempered elements in matrix algebras over group algebras of free groups, such as the fundamental groups of graphs.

If algebra $A$ is coherent over a Noetherian commutative ring $\kk$, and element $\Delta$ is well-tempered, then the homotope will be coherent too. Therefore, the category of finitely presented modules over the homotope is abelian.

We describe the representation theory of homotopes. Firstly, we concentrate our attention on  properties of functors and natural transformations between the categories of representations of algebra $A$ and the algebra $B$ constructed from $A$ via a well-tempered element in $A$.
Then we consider the derived categories and show that, in the case of a well-tempered element, the derived category of representations for $B$, $D(B-{\rm mod})$, satisfies the recollement data with the subcategory $D(\kk -{\rm mod})$ and the quotient category $D(A-{\rm mod})$. We  determine the gluing functor $D(A-{\rm mod})\to D(\kk -{\rm mod})$. The abelian category $B-{\rm mod}$ is the heart of the t-structure which is obtained by the gluing of the standard t-structures on the subcategory and the quotient category.

For a well-tempered $\Delta$, the category $D(B-{\rm mod})$ has a semiorthogonal decomposition into categories $D(\kk -{\rm mod})$ and $D(A-{\rm mod})$. Hence, the homotope construction, when viewed from the point of view of derived categories, looks more like a blow-up of a point on a surface, than the contraction of a divisor, as it actually is for a commutative $A$.


We study the representations of $B(\Gamma )$ for some graphs. The problem of describing the representations of minimal possible non-zero dimension is particularly interesting. In particular, if $B(\Gamma )=B(\Gamma_m(n))$, then the minimal dimension is $n$, and this is exactly the case that corresponds to orthogonal Cartan subalgebras in $sl(n)$. Another good example is when $\Gamma$ is a cyclic graph with $n$ vertices. Then the minimal possible dimension of representations for $B(\Gamma )$ is $n-2$ for suitable choice of parameters $s_{ij}$. 


The last theme that we touch in this survey is the relation of homotopes to the theory of perverse sheaves on stratified topological spaces. Since the category of representations of a homotope and the category of perverse sheaves both are the hearts of glued t-structures, it is reasonable to expect that in some situations the homotopes give the description of perverse sheaves. This is a hot topic, because explicit descriptions of perverse sheaves on stratified topological spaces are currently used for their categorifications called schobers \cite{KS}.

We make the first steps in this direction and show how to describe by means of homotopes the well-known category of perverse sheaves on the disc stratified by its central point and the complement to it. Our description is (at least superficially) different from the description by means of nearby and vanishing cycles, which might be of independent interest. 

Moreover, the category of perverse sheaves on a complex rational curve with a double point is equivalent to the category of finite dimensional representations $B(\Gamma )-\mod$, where $\Gamma$ is the cyclic graph with the generalised Laplace operator carefully chosen. This example can be regarded as a little bridge between Discrete Harmonic Analysis and the topology of stratified topological spaces. It would be interesting to compare this result with the description of microlocal sheaves on Riemann surfaces with ordinary double points proposed in \cite{BezKap}.

\section{Orthogonal Cartan subalgebras and algebraically unbiased projectors}
\label{sectprojectors}

\subsection{Orthogonal Cartan subalgebras}

Consider  a simple Lie algebra $L$ over an algebraically closed field of characteristic zero.
Let $K$ be the Killing form on $L$.
In 1960, J.G.Thompson, in course of constructing integer quadratic lattices with interesting properties, introduced the following definitions.

{\bf Definition.} Two Cartan subalgebras $H_1$ and $H_2$
in $L$ are said to be {\it orthogonal} if $K(h_1,h_2) = 0$ for all $h_1 \in H_1,
h_2 \in H_2$.

{\bf Definition.} Decomposition of $L$ into the direct sum
of Cartan subalgebras $L = \oplus^{h+1}_{i=1}
H_i$ is said to be {\it orthogonal} if
$H_i$ is orthogonal to $H_j$, for all $i \ne j$.

An intensive study of orthogonal decompositions has been undertaken since then (see the book \cite{KT} and references therein).
For Lie algebra $sl(n)$, A.I. Kostrikin, I.A, Kostrikin and V.A. Ufnarovsky have arrived to the following conjecture \cite{KKU}, called {\it Winnie-the-Pooh Conjecture} (cf. \cite{KT}, where, in particular, the name of the conjecture is explained by a wordplay in the
Milne's book in Russian translation).

\begin{Conjecture}
Lie algebra $sl(n)$ has an orthogonal decomposition if and only if $n = p^m$, for a prime number $p$.
\end{Conjecture}

The conjecture remains widely open. Even the non-existence of an orthogonal decomposition for $sl(6)$, for $n=6$ the first positive integer which is not a prime power, is not done yet. In this respect and in view of applications to Quantum Information Theory, which will be discussed below, it would be also interesting to find the maximal number of pairwise orthogonal Cartan subalgebras in $sl(n)$, for any given $n$, as well as to classify the maximal collections of such Cartan subalgebras modulo obvious symmetries.

We recall an interpretation and a generalization of the problem in terms of systems of minimal projectors and its relation to representation theory of Temperley-Lieb algebras and Hecke algebras of some graphs. This was discovered by the first author about 35 years ago (cf. \cite{KT}).

\subsection{Algebraically unbiased projectors}
\label{aup}

Let $V$ be a $n$-dimensional space over a field of characteristic zero.


Two minimal (i.e. rank 1) projectors $p$ and $q$ in $V$ are said to be {\it algebraically unbiased} if
\begin{equation}
\label{unb} {\rm tr}(pq) = \frac{1}{n}
\end{equation}
Equivalently, this reads as one of the two (equivalent) algebraic relations:
\begin{equation}\label{aunbias1}
pqp = \frac{1}{n}p,
\end{equation}
\begin{equation}\label{aunbias2}
qpq = \frac{1}{n}q.
\end{equation}

We will also consider {\it
orthogonal} projectors. Orthogonality of $p$ and $q$ is algebraically expressed as
\begin{equation}\label{orthog}
pq = qp = 0
\end{equation}

Two maximal (i.e. of cardinality $n$) sets of minimal orthogonal projectors $(p_1,...,p_n)$ and $(q_1,...,q_n)$  are said to be {\em algebraically unbiased} if
$p_i$ and $q_j$ are algebraically unbiased for all pairs $(i,j)$.


Let $sl(V)$ be the Lie algebra of traceless operators in $V$. The Killing form is given by the trace of the product of operators. A Cartan subalgebra $H$ can be extended to an abelian  subalgebra $H'$ in the matrix algebra by adding the identity operator $E$. The projectors of rank 1 in $H'$  is the maximal set of orthogonal projectors in $V$. Clearly, this set is unique. 
We say that these projectors are {\it associated} to $H$.

If $p$ is a minimal projector in $H'$, then ${\rm tr} (p)=1$, hence, $p-\frac 1nE \in H$. If projectors $p$ and $q$ are associated to orthogonal Cartan subalgebras, then
$$
{\rm tr}(p-\frac 1nE)(q-\frac 1nE)=0,
$$
which is equivalent to $p$ and $q$ to be algebraically unbiased.

Therefore, pairs of orthogonal Cartan subalgebras (or, in other words, orthogonal pairs) are in one-to-one correspondence with pairs of algebraically unbiased maximal sets of minimal orthogonal projectors. Similarly, orthogonal decompositions of $sl(n)$ correspond to $n+1$ of pairwise algebraically unbiased sets of minimal orthogonal projectors. In the analysis of the conjecture 1, it is worthwhile to consider unbiasedness not only for maximal sets of orthogonal projectors, but also for non complete orthogonal sets. Thus, we come to the problem of studying the sets of projectors where every pair satisfies either conditions (\ref{aunbias1}-\ref{aunbias2}) or (\ref{orthog}). Note here that this will lead us to the representation theory of reduced Temperley-Lieb algebras which we study in chapter 4.

Let us express the algebraic unbiasedness in terms of vectors and covectors. 
Let projectors $p$ and $q$ be given as:
$$
p=e\otimes x,\ \ q=f\otimes y ,
$$
where $e$ and $f$ are in $V$ and $x$ and $y$ are in $V^*$. The equations $p^2=p$ and $q^2=q$ imply:
\begin{equation}\label{projector}
(e, x)=1, \ \ (f, y)=1,
\end{equation}
where $(-,-)$ stands for the pairing between vectors and covectors.
Then the algebraic unbiasedness of $p$ and $q$ reads:
\begin{equation}\label{aunbias}
(x, f)(y, e)=\frac 1n .
\end{equation}
Orthogonality conditions (\ref{orthog}) reads:
\begin{equation}\label{ortho}
(x,f)=0, \ \ (y,e)=0.
\end{equation}

\subsection{Graph formulation of the problem}
\label{problem formulation}
The above discussion suggests to consider the following general problem on systems of projectors.

Assume we are given a simply laced graph $\Gamma$ with a finite number of vertices. Consider a finite dimensional vector space $V$. Assign a rank 1 projector in $V$ to every vertex of the graph. If two vertices are related by
an edge, then we require the corresponding two projectors to be algebraically unbiased. If there is no edge between the vertices, then we require the projectors to be orthogonal. The problem is to classify, for a given graph, all possible systems of projectors satisfying the required conditions modulo automorphisms of $V$.

The problem enjoys a {\em duality} which exchanges the roles of the space $V$ and its dual $V^*$. Indeed,
equations (\ref{aunbias1}), (\ref{aunbias2}), (\ref{orthog}) are invariant under the operation
of assigning to the vertices the adjoint projectors in the dual space $V^*$:
\begin{equation}
p\mapsto p^*.
\end{equation}
Thus, having a configuration in the space $V$, we obtain a dual configuration in $V^*$ and {\em vice versa}.

Rank 1 projectors are parameterized by the variety ${\mathbb P}(V)\times {\mathbb P}(V^*)\setminus D$, where $D$ is the incidence divisor. Indeed, every projector $p$ has the form $p=e\otimes x$, for $e$ a vector and $x$ a covector satisfying
$(e, x)=1$. There is the action of the torus $k^*$ on $V\times V^*$ by formula $e\mapsto \lambda e$, $x\mapsto \lambda ^{-1}x$, for $\lambda \in k^*$. Clearly, if $p=e\otimes x$ then $e$ and $x$  are determined uniquely up to the action of $k^*$.

Further, $V\times V^*$ is endowed with the standard symplectic form, which is invariant under the torus action. The locus of the solutions of equation
$(e, x)=1$ is the preimage of the unit under the moment map. Therefore, the quotient of the variety given by the equation modulo the $k^*$ action is an instance of the symplectic reduction. Obviously, the induced symplectic structure on ${\mathbb P}(V)\times {\mathbb P}(V^*)\setminus D$ coincides with Kostant-Kirillov-Lie bracket on the (co)adjoint orbit of minimal projectors of Lie algebra $gl(V)$.

We consider
the moduli space of solutions to the problem of classification of systems of projectors. In fact, the fine moduli is clearly a stack, because every configuration of projectors has the  obvious automorphism group $k^*$. Nevertheless, suitable conditions on configurations imply no extra automorphism.

\begin{Proposition}\label{minimal-conf}
Let $\Gamma$ be a connected graph. Assume that either the configuration of projectors is such that
the images of the projectors span the space $V$ or the intersection of kernels of projectors is trivial. Then the only automorphisms of the configuration are the scalars (i.e. the group of automorphisms is $k^*$).
\end{Proposition}
\begin{proof} Assume that the images of the projectors span the space $V$.
An automorphism of the configuration must preserve the image of every projector. If projectors $p$ and $q$ are algebraically unbiased, then for a vector $v$ in the image of $p$ we have:
$$
pqv=\frac 1n v,
$$
which shows, in particular, that $qv$ is a nonzero vector in the image of $q$. By applying to this equality any operator that commutes with $p$ and $q$, we see that this operator acts on $v$ and $qv$ via multiplication by the same scalar. Hence, it acts on the images of $p$ and $q$ by multiplication with the same scalar. Since the graph is connected, the scalar is the same for the images of all projectors. As the images of the projectors span $V$, it follows that the operator acts by a scalar. Using the duality, gives the fact in the case when the intersection of the kernels of the projectors is trivial, which proves the proposition.
\end{proof}


We say that a configuration is {\em minimal} if the images of the projectors span vector space $V$ and the intersection of the kernels of the projectors is zero.
Let ${\tilde {\cal M}}^n (\Gamma )$ be the moduli stack of minimal configurations in $k^n$. It follows from proposition \ref{minimal-conf} that it is a $k^*$-gerbe. We denote by ${\cal M}^n(\Gamma )$ the coarse moduli space of this gerbe. Clearly, it is the quotient of a subvariety in the cartesian product of copies of ${\mathbb P(V)}\times {\mathbb P(V^*)}\setminus D$, one copy for every vertex of the graph, modulo the action of $GL(V)$. Note that this presentation is not particularly convenient for calculation.

To every configuration, we can assign in a canonical way a minimal configuration at the price of reducing the dimension of the vector space. To this end, take the subspace generated by the images of all projectors. There is the induced configuration subordinated to the same graph in this subspace. If the intersection of the kernels of all projectors in the new configuration is nontrivial, then mode out this intersection. The resulting configuration is minimal.

We say that two configurations are $S$-{\em equivalent} if the corresponding minimal configurations are isomorphic. Consider the set ${\cal M}(\Gamma )$ of isomorphism classes of all minimal configurations for all possible dimensions $n$. We will endow it with the structure of an affine variety. It has a stratification with strata ${\cal M}^n (\Gamma )$ which parameterize the minimal configurations of given dimension $n$:
\begin{equation}\label{xdecom}
{\cal M}(\Gamma )=\bigcup_n {\cal M}^n(\Gamma )
\end{equation}

Equivalently, this variety parameterizes the $S$-equivalence classes of configurations in the vector space of dimension equal to the number of vertices in the graph.
Note that all this has a natural interpretation in the context of representation theory and the gluing of t-structures, which we will discuss below.

The symplectic geometry of the theory of unbiased projectors is discussed in \cite{BZh}.

\subsection{The moduli space of configurations as a torus}\label{calmod}

We will show that ${\cal M}(\Gamma )$ is in fact a torus $k^{*N}$. Let us describe a generating set of functions on it. Consider a cyclic path $\gamma$  of length $s$ in the graph $\Gamma$. Choose any orientation of the path and take the product $p_{\gamma}=p_1\dots p_s$ of projectors corresponding to the vertices of the path taken in any full order compatible with the cyclic order defined by the orientation (some vertices might recur). The trace $T_{\gamma}$ of the resulting operator does not depend obviously on the choice of the full order. We get a set of regular functions on the moduli space of configurations. The set is numbered by the homotopy classes of cyclic paths in $\Gamma$.

Since the operator $p_{\gamma}p_1$, if nonzero, has the same kernel and image as $p_1$ has, it is proportional to $p_1$. By checking the trace, we get:
$$
p_{\gamma}p_1=T_{\gamma}p_1.
$$
Let ${\hat {\gamma}}$ be the cyclic path $\gamma$ with inverse orientation and $p_{\hat {\gamma}}=p_s\dots p_1$, the operator  which corresponds to the full order inverse to the one chosen for $\gamma$. Then we have as above:
$$
p_1p_{\hat {\gamma}}=T_{{\hat \gamma}}p_1.
$$
Using this equalities, we get:

\begin{equation}\label{invloop}
T_{\gamma}T_{{\hat \gamma}}p_1=p_{\gamma}p_1p_{{\hat \gamma}}= \frac 1 {n^s} p_1.
\end{equation}
The latter equality is obtained by iteratively applying (\ref{aunbias1}) and (\ref{aunbias2}). Therefore:
\begin{equation}\label{invtrace}
T_{\gamma}T_{\hat \gamma}= \frac 1 {n^s}
\end{equation}
It is convenient to introduce normalized functions $S_{\gamma}$:
\begin{equation}\label{sgamma}
S_{\gamma}=n^{\frac12 |\gamma |}T_{\gamma},
\end{equation}
where $|\gamma |$ is the length of the cyclic path $\gamma$, i.e. the number of edges in $\gamma$.
Then equation (\ref{invtrace}) reads as a convenient cancellation law:
\begin{equation}\label{cancel}
S_{\gamma}S_{\hat \gamma}=1.
\end{equation}
Now we regard $\Gamma $ as a topological space, i.e. as a 1-dimensional CW-complex.
\begin{Theorem}\label{embedding}
\begin{itemize}
\item[(i)]
$S_{\gamma}$ depends on the homotopy class of the free loop $\gamma$ only;
\item[(ii)]
The assignment $\gamma \mapsto S_{\gamma}$ extends to a homomorphism
${\rm H}_1(\Gamma , {\mathbb Z})\to k^*$;
\item[(iii)] $S_{\gamma}$ does not depend on the choice of representation in the $S$-equivalence class.
\item[(iv)]
The map ${\cal M}(\Gamma )\to {\rm H}^1(\Gamma ,k^*)$ defined by $S_{\gamma}$'s is
bijective.
\end{itemize}
\end{Theorem}
\begin{proof} Given a cyclic path $\gamma$, there is a minimal free loop $\gamma_0$ in the graph which represents the homotopy class of $\gamma$. The cyclic path $\gamma$ can be contracted to $\gamma_0$ by elementary contraction, i.e. contractions where only paths along an edge in one direction and then immediately back, is contracted. Equations
(\ref{aunbias1}) and (\ref{aunbias2}) imply that $S_{\gamma}$ does not change under such contractions. This proves(i).

It is more convenient for us to postpone the proof for the other statements until we develop a more conceptional approach via Poincare groupoid (see section 7.3).


\end{proof}

Let $\Gamma_m(n)$ be the graph with $m$ rows and $n$ vertices in
each row, such that any two vertices are connected by an edge if and only if they are in different rows.

In view of what was explained in section \ref{aup}, a configuration of projectors subordinated to the graph $\Gamma_m(n)$ in the vector space $V$ of dimension $n$ gives a configuration of $m$ pairwise orthogonal Cartan subalgebras in $sl(n)$. The restriction on the dimension is crucial here. In the above picture, this corresponds to distinguishing the stratum on the torus which parameterizes representations of the minimal possible dimension. The moduli of configurations of orthogonal Cartan subalgebras is the quotient of ${\cal M}^n(\Gamma_m(n) )$ by the action of the product of $m$ symmetric groups $S_n$ that permute the vertices in the rows of the graph.


\subsection{Generalizations of the problem}

It is natural to put the problem on configurations of projectors into a broader context.
We can substitute the constant ${\frac 1n}$ at the right hand side of (\ref{aunbias1}) and (\ref{aunbias2}) by arbitrary non-zero constant $r\in k^*$. We say that two rank 1 projectors are $r$-{\em unbiased} if:

\begin{equation}\label{raunbias1}
pqp = rp,
\end{equation}
or, equivalently,
\begin{equation}\label{raunbias2}
qpq = rq.
\end{equation}

Further, when considering the problem on system of projectors subordinated to a graph, we can make $r$ to be  dependent on the edge. Then the initial data is a graph $\Gamma$ with labels $r_{ij}\in k^*$ assigned to edges $(ij)$ in the graph. The relations are:

\begin{equation}\label{unbgen1}
p_ip_jp_i = r_{ij}p_i,
\end{equation}
\begin{equation}\label{unbgen2}
p_jp_ip_j = r_{ij}p_j,
\end{equation}
if there is an edge $(ij)$ in the graph. We keep also the condition that $p_i$ and $p_j$ are orthogonal projectors
if there is no edge $(ij)$.

All what was said in section 2.4 holds true (with little correction) for this generalized version. In particular, formula (\ref{sgamma}) for $S_{\gamma}$ then reads:
\begin{equation}\label{sgamma1}
S_{\gamma}={\frac 1 {\sqrt{\prod r_{ii+1}}}}T_{\gamma},
\end{equation}
where the product under the square root is taken over all edges in $\gamma$ and $T_{\gamma }$ is the trace of the product of projectors along cyclic path $\gamma$.
The rest goes {\em mutatis  mutandis}.

The generalized version of Theorem (\ref{embedding}) holds true with the change of the wording that we leave to the reader.

Further, we can also generalize the problem by considering projectors of higher rank. Equations (\ref{unbgen1}) and (\ref{unbgen2}) are not equivalent conditions in this case. Hence, we should pose them both. It follows immediately from these equations that $p_i$ and $p_j$ have the same rank. Thus all projectors from the system are of the same rank if the graph is connected. We will reformulate this problem in terms of the representation theory of algebras $B(\Gamma )$, which we consider in section 4. Similar to the homological interpretation for the case of rank 1 projectors, the higher rank case is related to local systems of higher rank on the graph $\Gamma $ regarded as a topological space.

\subsection{Generalized Hadamard matrices.}

In this subsection, we will show how generalized Hadamard  matrices are related to orthogonal pairs of Cartan subalgebras in Lie algebra $sl(n)$.

Denote by $\cal M$ the set of $n\times n$ matrices  with non-zero entries over the field $k$.
A matrix $A = \{a_{ij}\}$ from $\cal M$ is said to be a {\it
generalized Hadamard matrix} if
\begin{equation}\label{hadamard}
\sum^{n}_{j=1}\frac{a_{ij}}{a_{sj}} = 0.
\end{equation}
for all $i \ne s$.

This condition can be recast by means of {\it Hadamard involution}
$h: {\cal M} \to {\cal M}$ defined by
\begin{equation}
h: a_{ij} \mapsto \frac{1}{na_{ji}}.
\end{equation}
\begin{Proposition} A is a generalized Hadamard matrix if and only if $A$ is invertible and $h(A) = A^{-1}$.
\end{Proposition}
\begin{proof} Indeed, (\ref{hadamard}) is equivalent to $A\cdot h(A)=1$.
\end{proof}

For any two Cartan subalgebras in a simple Lie algebra, one is always a conjugate of the other by an automorphism of the Lie algebra.
For the case of $sl(n)$, Cartan subalgebras are conjugated by an element of $GL_n$.
Then we can assume that the diagonal Cartan subalgebras is one of them, while the other one is conjugate to it by means of a nondegenerate matrix $A$.   
The normalizer consists of monomial matrices. Therefore, matrix $A$ is defined up to transformations
\begin{equation}\label{hadtrans}
A' = M_1 A M_2,
\end{equation}
where $M_1$ and $M_2$ are invertible monomial matrices.


\begin{Proposition}\cite{KT} Two Cartan subalgebras $H$ and $AHA^{-1}$ form an orthogonal pair in $sl(n)$ if and only if $A$ is a generalized Hadamard matrix.
\end{Proposition}

Recall that orthogonal pairs of Cartan subalgebras are associated to configurations of projectors governed by the graph $\Gamma_2(n)$. 
According to Theorem \ref{embedding}, the moduli of configurations in all dimensions is a $k^*$-torus of dimension $(n-1)^2$ (the rank of 1-st homology for graph $\Gamma_2(n)$).
This torus is identified with the quotient of $\cal M$
by the left and right multiplication of diagonal matrices. Equations (\ref{hadamard}) are easily seen to be invariant under this action.
They define the stratum of this torus which parameterizes representations of the minimal possible dimension, which is $n$ in this case.




\section{The Hermitian case}

\subsection{Mutually unbiased bases and configurations of lines in a Hermitian space}
The terminology of unbiased bases firstly appeared in physics \cite{Schw}, \cite{Ivo}. It is a unitary version of the algebraic unbiasedness introduced in section 2.
Here we will define them and later explain on an example how unbiased bases show up in Quantum Information Theory.


Let $V$ be an $n$ dimensional complex space with a fixed
Hermitian metric $\langle \ ,\ \rangle $. Two orthonormal Hermitian
bases $\{e_i\}$ and $\{f_j\}$ in $V$ are {\em mutually
unbiased} if, for all pairs $(ij)$, holds:
\begin{equation}\label{ef}
|\langle e_i, f_j \rangle|^2 =\frac 1n
\end{equation}

Consider the orthogonal projectors $p_i$ and $q_j$ corresponding to these bases, defined by:
$$
p_i(-)=e_i\otimes \langle - , e_i \rangle ,\ \ q_j(-)=f_j\otimes \langle - , f_j \rangle .
$$
Then condition (\ref{aunbias}) is satisfied for them, hence
the pairs of projectors $(p_i,q_j)$, for any $i,j$, are algebraically unbiased. Note that these operators are rank
1 Hermitian projectors, and, being such, are
defined by non-zero vectors in their images. We say that a pair of 
rank 1 projectors is {\em unbiased} if it is an algebraically unbiased pair of Hermitian projectors.

We can formulate a general problem about unbiased Hermitian projectors for graphs similar to that for algebraically unbiased projectors. Since a rank 1 Hermitian projector is determined  by its image, the one dimensional subspace (a line) in $V$, the configurations could be understood in terms of collections of points in ${\mathbb P}V$, but we will formulate it in terms of lines.

We fix again a simply laced graph $\Gamma $ with a finite number of vertices and a finite dimensional vector space $V$ of dimension $n$, now endowed with a Hermitian form. Assign a length 1 vector in $V$ to every vertex in $\Gamma$. If two vertices are connected by an edge, then we put condition
\begin{equation}\label{ef1}
|\langle e, f \rangle|^2 =\frac 1n
\end{equation}
on vectors $e$ and $f$ corresponding to the vertices. If the vertices are disconnected, then we require the vectors to be perpendicular. The problem is to classify all systems of vectors modulo linear automorphisms of $V$ and the products of $U(1)$'s corresponding to change of phases of all vectors assigned to vertices.

There is a Hermitian version of the generalised algebraic problem. 
To this end, we consider a {\em full} graph $\Gamma$ together with labels $r_{ij}$, $0\le r_{ij}\le 1$, on all edges $(ij)$. The problem is to find all, up to linear automorphisms of $V$, configurations of lines assigned to vertices of the graph and satisfying
\begin{equation}
|\langle e_i, e_j \rangle|^2 =r_{ij},
\end{equation}
for any choice of length 1 vectors $e_i$ and $e_j$ in the lines corresponding to vertices $i$ and $j$ of the graph.
Vanishing $r_{ij}=0$ corresponds to the absence of edge between $i$ and $j$ in the graph in the algebraic version of the problem (in which case the condition was given by two equations (\ref{ortho}), instead of one in the Hermitian case).

Note that all our conditions are just to fix particular angles between lines in the configuration.
More generally, one can consider subspaces of fixed dimension in the Hermitian space $V$ and fix angles between them.
This corresponds to higher rank representations of algebra $B_r(\Gamma )$ (see below).

\subsection{Moduli of configurations of lines and complex Hadamard matrices}
When the base field is the complex numbers we can regard the algebraic unbiasedness as a complexification of the unbiasedness. To be more precise, consider the space ${\cal M}(\Gamma )$ of all configurations of algebraically unbiased projectors for graph $\Gamma$. 
It is an algebraic subvariety in the cartesian product of several copies of ${\mathbb P}(V)\times {\mathbb P}(V^*)\setminus D$, one copy for every vertex of $\Gamma$.

Now fix a Hermitian form on $V$. The Hermitian involution gives a new duality on the set of algebraic configurations:
\begin{equation}
p_i\mapsto p_i^{\dag }.
\end{equation}
The duality induces an anti-holomorphic involution on ${\cal M}(\Gamma )$. It is easy to see that the involution takes $S_{\gamma }$ to ${\bar S_{{\hat \gamma}}}$. 

Clearly, the involution preserves all strata ${\cal M}^n({\Gamma })$, in particular, it takes any minimal configuration to a minimal one. 



Since mutually unbiased bases are algebraically unbiased, they give examples of orthogonal Cartan subalgebras in $sl(n)$. Given $m$ pairwise mutually unbiased bases
${\cal B}_1,{\cal B}_2, . . . ,{\cal B}_{m}$ in a Hermitian space $ V$, we obtain $m$ pairwise orthogonal Cartan
subalgebras $H_1,H_2, . . . ,H_{m}$ in $sl(n)$. 
In particular, a collection of $n+1$ mutually unbiased bases in a Hermitian vector space of dimension $n$ gives rise to an orthogonal
decomposition of $sl(n)$. This fact was noticed 
in \cite{BTSW}.

Let $\cal B$ be an orthonormal basis in $\CC^n$. Matrix $A = (a_{ij})$ is said to be {\it complex Hadamard} if bases $\cal B$ and $A({\cal B})$ are mutually unbiased.
Let $A$ and $C$ be a complex Hadamard matrices. We will say that $A$ is equivalent to $C$ if $A = M_1 C M_2$ for some unitary monomial matrices $M_1, M_2$.

The relation between complex Hadamard matrices and generalized Hadamard ones is as follows:
$A$ is a complex Hadamard $(n\times n)$ matrix if and only if $A$ is a generalized Hadamard matrix and $|a_{ij}| = {\frac 1n}$.

The problem of classification of generalised (respectively complex) Hadamard matrices is in fact the problem of classification of orthogonal pairs of Cartan subalgebras in $sl(n, \CC)$ (respectively of mutually unbiased bases in $\CC^n$). This problem is solved still only for $n \le 5$ (see  \cite{KKU2}, \cite{Haa}): it turns out that when $n = 2,3,5$ there is only finite number of Hadamard matrices, while for $n = 4$ the matrices are parameterized by a curve. In the case $n = 6$, there was a long standing conjecture on the existence of a 4-dimensional family of the matrices (see \cite{MS}, \cite{Szo}, \cite{TZ}). The authors of the present paper has proven this conjecture in \cite{BZh2}.
Nevertheless this does not provide a complete classification of generalised (or complex) Hadamard matrices in dimension 6, because beyond the 4-dimensional family there is at least isolated (not in any family) generalized Hadamard matrices. Also there was a conjecture formulated by S. Popa (see \cite{Po}) 
on the finiteness of the number of complex (generalised) Hadamard matrices when the dimension of the vector space is a prime number. This conjecture has been falsified by his student M. Petrescu \cite{Pet}, who has 
constructed via computer simulations 1-dimensional families in dimension 7, 13, 19, 31, 79. In the case of dimension 7, R. Nicoara has noticed in \cite{Nic} that projectors corresponding to mutually unbiased bases satisfy a commutator relation. This observation has got a further development in \cite{KZh1}, \cite{KZh2}.  

\subsection{Quantum protocols}
\label{protocols}

Consider a finite dimensional quantum mechanical system for which
 the space of states is an $n$ dimensional complex vector space $V$. Recall that {\em states} of the system are vectors
in $V$ considered up to multiplication by non-zero constants.
{\em Observables} of the system are Hermitian operators in $V$. For the sake of simplicity, choose an
observable which has
pairwise distinct real eigenvalues $\lambda_i, i=1,\dots, n,$ in the standard orthonormal basis $\{e_i\}_{i\in [1,n]}$. Due to the basic
quantum mechanical principles, the {\em measurement} of this
observable when the system is in a normalised state $f$ (i.e. $|f|=1$) returns us one of the eigenvalues $\lambda_i$ of the observable with the probability equal to 
$|\langle
e_i, f \rangle|^2$. 

Another important result of the measurement is that after the measurement the state of the system transfers into the eigenvector $e_i$. This leads to many problems in Quantum Theory, because the state of the system is changed after the measurement, and to repeat the same measurement with the system is, in general, no longer possible.

The condition of mutual unbiasedness (\ref{ef}) frequently appears in problems of Quantum
Information Theory. Here is one of the typical examples.

Let Alice transmits quantum information to Bob by using a quantum
channel, which is a chain of quantum particles whose states are described by vectors in a vector  space $V$. A protocol for transmitting quantum
information is a set of pairwise mutually unbiased bases in $V$.
Alice and Bob has arranged the sequence of pairwise mutually
unbiased bases which they will use in their communication, say
$\{e_i\}$, $\{f_j\}$, $\{g_k\}$, etc. Alice sends states of the
systems: firstly, one of the vector of the basis $\{e_i\}$, then a
vector of the basis $\{f_j\}$, etc. Possible transmitted string
might look as $(e_2, f_5, g_3, \dots )$. The actual information is
contained in the indices $2, 5, 3, \dots $. 

Now assume also the presence
of an eavesdropper, Eve. Since the protocol is an open information,
Eve knows the set of bases but she does not
know the sequence of them that Alice and Bob arranged among themselves. Eve
will get no information even probabilistically if she uses a wrong
basis from the protocol when trying to determine a transmitted state. Indeed, if
Alice sends one of the $e_i$'s and Eve measures an observable which,
as a Hermitian operator, is diagonalisable in the basis $\{f_j\}$, then
she will get one of the eigenvalues  of this observable with equal
probability $\frac 1n$. 

Clearly, the bigger is the number of mutually
unbiased bases in the protocol, the more secure is the protocol against the attempts of interception. Also protocols of quantum channels are widely used for the protection against obstacles or 'noises' in the channel. Famous protocol BB84 was the first protocol which in essence utilises unbiased bases in the simplest situation of a 2-dimensional vector space of states of a quantum qubit \cite{BB84}. There is one up to conjugation maximal set of three bases. Under the identification with Cartan subalgebras in $sl(2, {\mathbb C})$, this set corresponds to Cartan subalgebras spanned by Pauli matrices.

When the space of quantum states of transmitted particles has higher dimension, the problem of classification of mutually unbiased bases arises. It follows from the correspondence of mutually unbiased bases and orthogonal Cartan subalgebras that the number of such bases is not greater than $n+1$, where $n$ is the dimension of the state space. The demand of describing   maximal sets of mutually unbiased bases explains the
relevance of the Winnie-the-Pooh problem to Quantum Information
Theory.

There are many other instances when mutually unbiased bases are used in physics, see 
\cite{FM}, \cite{Rus}, \cite{Woo}, \cite{Eng}, \cite{Hua}, \cite{Vai}, \cite{Got}, \cite{Cal}, etc.

%

\section{Algebra $B(\Gamma)$, Hecke algebra and Poincare groupoid of graph}

\subsection{Algebra $B(\Gamma)$.}

The above discussion of the problem on configurations of projectors motivates the study of representation theory for algebra $B({\Gamma})$, which we introduce here. Under some specialization of parameters, this
algebra becomes a quotient of more familiar Temperley-Lieb algebra of the 
graph $\Gamma$. The latter is, in its turn, a quotient of Hecke algebra of
the graph.

Let $\Gamma$ be a simply laced graph.  Denote by $V(\Gamma)$ and $E(\Gamma)$ the sets
of vertices and edges of the graph.
Let $\kk$ be a commutative ring, Choose a set of invertible elements $\{s_{ij}\}$ in $\kk$, where $(ij)$ runs over the set of all edges of the graph (i.e. $s_{ij}=s_{ji}$). For example, one can take the universal ring $\kk =k[\{  s_{ij}\},
\{ s_{ij}^{-1}\} ]$, where $k$ is a field of
characteristic zero.
We define
algebra $B(\Gamma )$ as a unital algebra over $\kk$. Generators $x_i$ of $B(\Gamma )$ are numbered by the
vertices $i$ of $\Gamma$. They subject relations:
\begin{itemize}
\item{$x^2_i = x_i$, for every $i$ in $V(\Gamma )$,}

\item{$x_ix_jx_i = s_{ij}^2x_i$, $x_jx_ix_j = s_{ij}^2x_j$, if $i$ and $j$ are adjacent
in $\Gamma$,}

\item{$x_ix_j = x_jx_i = 0$, if there is no edge $(ij)$ in $\Gamma$.}

\end{itemize}
We define $B^+(\Gamma )$ as the augmentation ideal in $B(\Gamma )$ generated by all $x_i$'s.
Note that automorphisms of the graph induce automorphisms of $B(\Gamma)$.

A configuration of projectors $p_i$'s considered in section \ref{sectprojectors} can be understood as a representation of $B(\Gamma )$:
$$
x_i\mapsto p_i,
$$
if we pose
$$
r_{ij}=s_{ij}^2.
$$
It will be convenient for us to keep the square roots of $r_{ij}$ as basic parameters.

A path in the graph is a sequence of vertices
$$
\gamma = (i_0, \dots , i_t),
$$
where $i_l=i_{l+1}$ or $(i_li_{l+1})\in E(\Gamma )$, for $0\le l \le t-1$. To such a path, we assign an element $x_{\gamma} \in B(\Gamma )$:
$$
x_{\gamma}= x_{i_0}\dots x_{i_t}.
$$
We assign 1 to the empty path.

Path $\gamma$ is said to be {\em contracted} if 
$$
i_l\ne i_{l+1},\ {\rm for}\ 0\le l \le t-1,\ {\rm and}\ i_l\ne i_{l+2},\ {\rm for}\ 0\le l \le t-2,
$$
or if it is an empty path.
If $\gamma$ is not contracted, i.e. $i_l=i_{l+1}$ or $i_l=i_{l+2}$ for some $l$, then we define its {\em elementary contraction} as a new path with vertex $i_{l+1}$
removed from $\gamma$. We can obtain other contractions of the path by iterating minimal contractions. Note that all contractions of $\gamma$ preserve the homotopy class
of $\gamma$ among the paths with fixed starting and ending points.
Every homotopy class of paths with fixed ends contains a unique contracted path, which we call the {\em minimal contraction}. The minimal contraction can be achieved by a sequence of elementary contractions.

\begin{Proposition}\label{basisbg}
The set of elements $x_{\gamma}$ where $\gamma$ runs over the set of all contracted paths in $\Gamma$, is a $\kk$-basis for algebra $B({\Gamma })$.
\end{Proposition}
\begin{proof}
The defining relations for $B({\Gamma})$ imply that $x_{\gamma}$ remains the same up to invertible multiplier after an elementary contraction and that the $\kk$-linear span of $x_{\gamma}$'s is $B({\Gamma})$. 

It remains to show that these elements are linearly independent. To this end, note that any element in the ideal of relations for $B(\Gamma )$ among $x_i$'s is a $\kk$-linear combination of two classes of relations:
\begin{itemize}
\item $x_{i_1}\dots x_{i_t}=0$ when, for some $l$, the pair $(i_l, i_{l+1})$ is not an edge of the graph.
\item $x_{\gamma}=\lambda_{\gamma \gamma '} x_{\gamma '}$ where $\gamma'$ is an elementary contraction of $\gamma$ and $\lambda_{\gamma \gamma '} \in \kk^*$;
\end{itemize}
The first class shows that it is enough to consider only monomials corresponding to paths in the graph.
The second class of relations is divided into groups, where each group consists of relations among monomials $x_{\gamma}$
where $\gamma$ is in a homotopy class of paths with fixed starting and ending points. Hence, it is enough to check that monomials from every homotopy class of paths span a 1-dimensional space in $B({\Gamma })$.
Denote by $q_{ij}$ the difference between multiplicities of edge $(ij)$ in $\gamma$ and its contraction $\gamma '$. Then one can easily check by induction on the number of elementary contractions required to pass from $\gamma$ to $\gamma '$ the formula:
$$
x_{\gamma}=\prod s_{ij}^{q_{ij}}x_{\gamma '}.
$$
By taking $\gamma '$ to be the minimal contraction in the homotopy class of $\gamma$, we get that any monomial has an expression as the monomial of the contracted path with a unique invertible multiplier in $\kk$. Thus, the space spanned by the monomials in one homotopy class is indeed 1-dimensional. This finishes the proof of the theorem.
\end{proof}

We will write $B_s(\Gamma )$ when we need to specify the values of $s=\{s_{ij}\}$.
We shall consider the representation theory of these algebras in section 7.

\subsection{Temperley-Lieb algebras and Hecke algebras of graphs}
\label{tlhecke}

Consider the special case when $s_{ij}=s$ for all edges $(ij)$ of the graph. To describe the relation to Hecke algebra, note that the standard Temperley-Lieb algebra $TL(\Gamma)$ is
defined similarly to $B(\Gamma )$, with the last relation replaced by $x_ix_j =
x_jx_i$ when there is no edge connecting $i$ and $j$. Hence, algebra
$B(\Gamma)$ for the case $s_{ij}=s$ is a quotient of $TL(\Gamma)$.


Recall that Hecke algebra $H(\Gamma)$ of graph $\Gamma$ is the unital algebra over $k[q, q^{-1}]$ generated by
elements $T_i, i \in V(\Gamma)$ subject the following relations:
\begin{itemize}
\item{$T_i T_j T_i = T_j T_i T_j$, if $(i,j) \in E(\Gamma)$,}
\item{$T_i T_j = T_j T_i$, otherwise,}
\item{$(T_i + 1)(T_i - q) = 0$ for any $i$.}
\end{itemize}
We obtain $H_q(\Gamma)$, an algebra over $k$, by specializing
$q$ to a non-zero value. For the case $\Gamma$ is  Dynkin graph of
type $A_n$, algebra $H_q(\Gamma)$ is known to
be a $q$-deformation of the group algebra for the permutation group $S_{n+1}$.
Its representation theory is useful in constructing polynomial
invariants of knots (cf. \cite{Lick}, \cite{Jon}). If $q$ is not a root of unity,
$H_q(\Gamma )$ is isomorphic to the group algebra of for $S_{n+1}$.

There is a homomorphism $H(\Gamma ) \to B(\Gamma )[q]$, where $B(\Gamma )$
is extended by the central element $q$, a root of the equation:
$$
q+2+q^{-1}=s^{-2}.
$$
The homomorphism is defined by mapping:
$$
T_i \mapsto (q+1)x_i - 1.
$$

One can easily see that via this homomorphism $B(\Gamma)$ becomes isomorphic to the quotient of $H(\Gamma)$ by relations:
\begin{itemize}
\item{$-1+T_i+T_j-T_i T_j - T_j T_i + T_iT_jT_i = 0$, if $(ij)$ is an edge;}
\item{$-1+T_i + T_j - T_i T_j = 0$, otherwise.}
\end{itemize}


Let us note that
if $\Gamma$ is a graph of type $A_n$ and $q$ is not a root of unity,
algebra $B(\Gamma)$, being a quotient of $H_q(\Gamma)$, is the direct sum of
the matrix algebra of the operators in the standard $n$-dimensional
representation of the permutation group $S_{n+1}$ and the field $k$. Note that $H_q(\Gamma )$ contains the whole representation theory of the symmetric group, while Temperley-Lieb algebra does the representations
corresponding Young diagrams with up to two rows.

Thus the representation theory of $B(\Gamma )$ is very simple for graph $A_n$.
The reason is that the homology of the graph is trivial. The measure of complexity of the representation theory for $B(\Gamma )$ is the first homology of the graph, as we will see later.

\subsection{Poincare groupoids of graphs}\label{Poincare}
It turns out that the representation theory for $B(\Gamma )$ is
closely related to that of Poincare groupoid of the graph.

Let again $\Gamma$ be a simply-laced graph. Consider it as a topological space. Let $\pg$ be the Poincare
groupoid of graph $\Gamma$, i.e. a category with objects vertices of the graph
and morphisms homotopic classes of paths. Composition of morphisms is given by concatenation of paths.

Let $\kk $ be a commutative ring. Denote by $\kk\Gamma $ the
algebra over $\kk $
with a free $\kk$-basis numbered by morphisms in
$\pg$ and multiplication induced by concatenation of paths (when it makes sense, and zero when it does not). Let $e_i$ be the element of $\kk\Gamma$ which is the constant path in vertex $i$. Any oriented edge $(ij)$ can be interpreted as
a morphism in $\pg$, hence it gives an element $l_{ij}$ in $\kk\Gamma$. These are the generators.
The defining relations are:
\begin{itemize}
\item{$e_i e_j = \delta_{ij} e_i,\ \ e_i l_{jk} = \delta_{ij} l_{ik},\ \ l_{jk} e_i = \delta_{ki} l_{jk}$;}
\item{$l_{ij} l_{ji} = e_i,\ \ l_{ji}l_{ij} = e_j,\ \ l_{ij}l_{km} = 0$, if $j \ne k$.}
\end{itemize}

We consider $\kk\Gamma$ as an algebra with unit:
$$
1=\sum_{i\in V(\Gamma )}e_i.
$$



Let $\gamma\in \pg$ be a path in $\Gamma$. Denote by $l_{\gamma}$ the element in $\kk\Gamma$ corresponding to this path. There is an involutive anti-isomorphism $\sigma :\kk\Gamma\to \kk\Gamma ^{opp}$ defined by
\begin{equation}\label{dualkg}
\sigma(l_{\gamma})=l_{{\hat \gamma}},
\end{equation}
where ${\hat \gamma}$ is the path inverse to $\gamma$. It implies a duality, i.e. an involutive anti-equivalence, $D: \kk \Gamma -{\rm mod}_{fd}\simeq \kk \Gamma -{\rm mod}^{opp}_{fd}$, on the category $\kk \Gamma -{\rm mod}_{fd}$ of finite dimensional $\kk \Gamma$-representations, if $\kk=k$ is a field. If $\rho :k\Gamma \to {\rm End}(V)$ is a representation, then the dual representation $D(\rho ):k\Gamma \to {\rm End}(V^* )$ is defined, for $l\in k\Gamma$, by:
\begin{equation}\label{dualforkg}
D(\rho ) (l)=\rho (\sigma (l))^*.
\end{equation}

Let $\Gamma$ be in addition a connected graph. Then the category of representations for Poincare groupoid and that
for the fundamental group of the graph are equivalent.
To see this, fix $t \in V(\Gamma)$. Denote by $\kk[\pi(\Gamma,t)]$
the group algebra of the fundamental group $\pi(\Gamma,t)$. Consider
projective $\kk\Gamma$ - module $P_t = \kk\Gamma e_t$. Clearly, $P_t$ is a $
\kk\Gamma$ - $\kk[\pi(\Gamma,t)]$ - bimodule. Note that $P_t$ are isomorphic as left $\kk \Gamma$-modules for all choices of the vertex $t$. Indeed, the right multiplication by an element corresponding to a path that connects $t_1$ with $t_2$ would give an isomorphism of $P_{t_1}$ with $P_{t_2}$.

\begin{Proposition}\label{moritakg} Bimodule $P_t$ induces a Morita equivalence between $\kk\Gamma$ and $\kk[\pi(\Gamma,t)]$.
Thus, the categories $\kk\Gamma - {\rm mod}$ and $\kk[\pi(\Gamma,t)] - {\rm mod}$ are
equivalent. Moreover, algebra $\kk\Gamma$ is isomorphic to the matrix algebra over $\kk[\pi(\Gamma,t)]$, with the size of (square) matrices equal to $|V(\Gamma )|$.
\end{Proposition}

The functors that induce equivalence between categories $\kk\Gamma - {\rm mod}$ and $\kk[\pi(\Gamma,t)] - {\rm mod}$ are:
\begin{equation}
\label{functors} V \mapsto P_t \otimes_{\kk[\pi(\Gamma,t)]} V,\ \ W \mapsto
{\rm Hom}_{\kk\Gamma}(P_t, W).
\end{equation}



Let us construct an isomorphism $\kk\Gamma\to {\rm Mat}_n(\kk[\pi(\Gamma,t)])$. To this end, fix a system of paths $\{\gamma_i\}$ connecting the vertex $t$ with every vertex $i$ (one path for every vertex). For any element $\pi\in \kk[\pi(\Gamma,t)]$, consider an element  $\gamma^{-1}_i\pi\gamma_j$ in $\kk\Gamma$. The homomorphism is defined by the assignment:
$$
\gamma^{-1}_i\pi\gamma_j\mapsto \pi\cdot E_{ij},
$$
where $E_{ij}$ stands for the elementary matrix with the only nontrivial entry $1$ at $(ij)$-th place. An easy check shows that this is indeed a well-defined ring isomorphism.

Let $\kk =k$ be a field.
Since the fundamental group $\pi(\Gamma,t)$ is free, the equivalence implies that the homological
dimension of category $k\Gamma - {\rm Mod}$ is 0 if the graph is a tree and 1 otherwise.

The equivalence takes the duality functor (\ref{dualforkg}) for $k\Gamma -{\rm mod}$ into the standard duality $W\mapsto W^*$ for representations of the group $\pi(\Gamma,t)$.

\section{Homotopes}

In this section we introduce the central mathematical object of this paper, the homotope. It is a purely formal algebraic construction of a new algebra, given an algebra and an element in it. Despite of simplicity of the construction, the representation theory of homotopes is rich. In particular, homotopes will allow us to describe algebraically the categories of perverse sheaves on some stratified topological spaces.


\subsection{Definition of homotopes}
\label{changemult}
Let $A$ be an algebra over a commutative ring $\kk$ and $\Delta$ an element in $A$. Consider the non-unital algebra, $A_{\Delta}$, with the same space as $A$ but with new multiplication defined by the formula:
\begin{equation}
a \cdot_{\Delta} b = a\Delta b.
\end{equation}
Algebra $A_{\Delta }$ is called {\em homotope of} $A$ (cf. \cite{J}).
By adjoining the unit to algebra $A_{\Delta}$, we get a unital algebra $\widehat{A}_{\Delta}=\kk \cdot 1\oplus A_{\Delta}$, which we will call by {\em unital homotope of} $A$.
Define two homomorphisms:
\begin{equation}
\psi_{1},\psi_{2}:A_{\Delta} \to A,
\end{equation}
by mappings:
\begin{equation}
\label{hom}
\psi_1: a \mapsto a\Delta , \ \ \psi_2:a \mapsto \Delta a.
\end{equation}
Note that
\begin{equation}
{\rm Im}\psi_1 = A\Delta,\ \ {\rm Im}\psi_2 = \Delta A,
\end{equation}
i.e. the images of the non-unital homotope $A_{\Delta}$ are the one-sided ideals in $A$ generated by $\Delta$.
We use the same notation $\psi_1$ and $\psi_2$ for the extensions to homomorphisms of unital algebras:
$$
\psi_{1},\psi_{2}:{\widehat A}_{\Delta} \to A,
$$
Note that $A_{\Delta}$ is the augmentation ideal, hence an $\widehat{A}_{\Delta}$-bimodule.
The left multiplication in $A$ commutes with right multiplication in $A_{\Delta}$ and vice versa.
We use the left and right multiplication in $A$ to endow $A_{\Delta}$ with the structure of $A$-bimodule. Clearly,
as a right or left $A$-module, $A_{\Delta}$ is free of rank 1.
The left $\widehat{A}_{\Delta}$-module structure on $A_{\Delta}$ coincides with the structure obtained by pulling back of the left $A$-module structure along morphism $\psi_1$. Similarly, the right $\widehat{A}_{\Delta}$-module structure is the pull-back of the right $A$-module structure along morphism $\psi_2$.

We use notation $B=\widehat{A}_{\Delta}$ and $B^+=A_{\Delta}$.
Denote by ${}_{\psi_i}A_{\psi_j}$ the $B$-bimodule structure on $A$ with left and right module structures defined by $\psi_i$ and $\psi_j$ respectively. The identification of $B^+$ with $A$ implies an isomorphism of $B$-bimodules :
\begin{equation}
B^+ \cong {}_{\psi_1}A_{\psi_2}
\end{equation}

If $\Delta$ is an invertible element in $A$, then $\psi_1$ and $\psi_2$ are isomorphisms. Thus, $A_{\Delta}$ is a unital algebra with the unit ${\Delta }^{-1}$. In this case,  ${\widehat A_{\Delta}}$ is obtained from a unital algebra, isomorphic to $A$, by adjoining a new unit $1_B$. This implies that algebra ${\widehat A_{\Delta}}$ is a product of algebras $A_{\Delta}\simeq A$ and $\kk$. Algebras $A_{\Delta}$ and $\kk$ are subalgebras in ${\widehat A_{\Delta}}$, where $\kk$ is embedded into ${\widehat A_{\Delta}}$ via $\kk (1_B-{\Delta }^{-1})$ (in order to annihilate $A_{\Delta}$). Thus, the construction of the homotope is of interest only when $\Delta $ is not invertible.

\subsection{An analogy: homotopes as affine contractions}

It is well-known that the contraction of a projective curve in an algebraic surface is possible only under suitable assumptions. For example, a smooth projective curve can be contracted (even in a formal neighborhood) only if its self-intersection number is negative. In the affine geometry, the picture is drastically different: any affine closed subscheme in an affine scheme can be contracted. Moreover, all colimits over finite diagrams in the category of affine schemes exist. Indeed, the category of affine schemes is dual to the category of commutative algebras, where limits exist. Though one should keep in mind that a  colimit of finitely generated commutative algebras may no longer be a finitely generated algebra. For instance, the contraction of an affine line in the affine plan is the spectrum of an infinitely generated algebra.

The homotope construction can be understood as a noncommutative version of the divisorial contraction in the affine algebraic geometry. Indeed, assume we have a commutative algebra $A$ and an element $\Delta$ in it, which is not a zero divisor in $A$. For this case, the maps $\psi_1$ and $\psi_2$ coincide and identify $B^+$ with the ideal in $A$ generated by $\Delta$. Then algebra $B$ is the limit over the diagram:

\begin{equation}\label{contractiondiagram1}
\xymatrix{& & A\ar[d]
\\
 & \kk\ar[r] & A/\Delta A}
\end{equation}

Since the dual of this diagram presents the contraction of ${\rm Spec}A/ \Delta A\subset {\rm Spec}A$ to the point ${\rm Spec}\kk$:
\begin{equation}\label{contractiondiagram2}
\xymatrix{{\rm Spec}A/\Delta A\ar[d]&\ar[r] & {\rm Spec}\kk\ar[d]
\\
 {\rm Spec}A& \ar[r] &{\rm Spec}B }
\end{equation}
we can regard homotope as the contraction of the principal divisor defined by $\Delta$. In section \ref{genrepbg} we dsicuss an example of contraction of the first infinitesimal neighborhood of a point
on the affine line, which results in the cusp singularity.

\subsection{Discrete Laplace operator and algebra $B(\Gamma)$ as a homotope of Poincare groupoid.}\label{sectionlaplace}


Let  $\kk$ be a ring which contains a set of invertible elements $\{s_{ij}\}$. For instance, a possible choice of $\kk$ is the ring of Laurant polinomials: $\kk=k[\{ s_{ij}\} ,\{ s_{ij}^{-1}\} ]$.

Consider the {\em generalized Laplace operator of a graph $\Gamma$}, i.e. an element $\Delta$ in the algebra $\kk\Gamma$ of Poincare groupoid of the graph:
\begin{equation}\label{deltaformula}
\Delta =1+ \sum s_{ij}l_{ij},
\end{equation}
where the sum is taken over all oriented edges.

Consider algebra $\kk \Gamma_{\Delta}$, the unital homotope of $\kk\Gamma$ constructed via the element $\Delta$.
Denote by $x_i$'s the elements in $\kk{\Gamma}_{\Delta}$ that correspond to $e_i$'s in $\kk{\Gamma }$.
The following theorem implies the realisation of $B(\Gamma)$ as the unital homotope of the Poincare groupoid $\kk\Gamma$.
\begin{Theorem}\label{bgammaviagr}
There is a unique isomorphism of non-unital algebras:
\begin{equation}
B^{+}(\Gamma) \cong \kk{\Gamma}_{\Delta},
\end{equation}
and hence, an isomorphism of unital algebras:
\begin{equation}
B(\Gamma) \cong \widehat{\kk{\Gamma}_{\Delta}}
\end{equation}
that takes $x_i$ into $e_i$.
\end{Theorem}
\begin{proof}
Firstly,
$$
x_i^2=e_i \cdot_{\Delta} e_i=e_i \Delta e_i=e_i=x_i.
$$
Further, we have:
$$
x_ix_j=e_i \cdot_{\Delta} e_j = e_i \Delta e_j = s_{ij} l_{ij},
$$
for $(ij)\in E(\Gamma )$, and
\begin{equation}\label{xixj0}
x_ix_j=0,
\end{equation}
when $(ij)$ is not an edge of $\Gamma$.

We also have:
\begin{equation}\label{llmult}
l_{ij}\cdot_{\Delta} l_{jk}= l_{ij}l_{jk}
\end{equation}
for edges $(ij)$ and $(jk)$.
Hence,
\begin{equation}\label{xl}
x_ix_jx_i=s_{ij}^2l_{ij}l_{ji}=s_{ij}^2x_i.
\end{equation}
Thus, we have checked the defining relations for $B(\Gamma )$. In other words, we have constructed a homomorphism
$$
B(\Gamma )\to \widehat{\kk{\Gamma}_{\Delta}},
$$
which we need to show to be an isomorphism.
The above equations imply that for every path $\gamma =(i_1\dots i_l)$, $l\ge 2$, in the graph, we have
\begin{equation}\label{xxll}
x_{i_1}\dots x_{i_l}=x_{i_1}x_{i_2}^2\dots x_{i_{l-1}}^2x_{i_l}\\
=s_{i_1i_2}\dots s_{i_{l-1}i_l}l_{i_1i_2}{\cdot}_{\Delta}\dots \cdot_{\Delta}l_{i_{l-1}i_l}\\
=s_{i_1i_2}\dots s_{i_{l-1}i_l}l_{i_1i_2}\dots l_{i_{l-1}i_l}.
\end{equation}
where we consider multiplication in $\kk \Gamma_{\Delta}$ at the left hand side and that in $\kk \Gamma $ at the right hand side. To get the last equality we use iteratively (\ref{llmult}) and the fact that the left multiplication in $\kk \Gamma_{\Delta}$ commutes with the right multiplication in $\kk \Gamma $. 

Since the set of elements $l_{i_1i_2}\dots l_{i_{l-1}i_l}$ where $(i_1, \dots , i_l)$ run over the set of contracted paths in $\Gamma$ is a $\kk$-basis in $\kk\Gamma$,
the bijectivity of the homomorphism follows from theorem \ref{basisbg}, which proves the theorem.
\end{proof}


By the difinition, a {\em tail} is a vertex in the graph with valency 1.

\begin{Proposition}
\label{divisor of zero}
Let $\Gamma$ be a non-empty connected graph with no tail
and $\kk$ a commutative ring with no zero divisor. Then $\Delta$ is neither left nor right zero divisor in $\kk\Gamma$.
\end{Proposition}
\begin{proof}
We will argue by contradiction: 
assume that $z\in \kk\Gamma$ is such that $z\Delta =0$. Recall that $\kk\Gamma$ is a direct sum of right modules $P_t$, where $t$ runs over all vertices of the graph. Then we have a decomposition $z=\sum z_i$. Every component $z_t$ satisfies  $z_t\Delta =0$. Thus, we reduced the problem to the case when $z\in P_t$ for some $t\in V(\Gamma )$. Consider the universal covering graph ${\tilde \Gamma}$ with a lift ${\tilde t}\in V({\tilde \Gamma})$ of the vertex $t$. Then we have a unique lift ${\tilde z}\in \kk {\tilde \Gamma}$ of $z$ in the groupoid algebra of the universal cover. Clearly, ${\tilde z}{\tilde \Delta}=0$, where ${\tilde \Delta}$ is the Laplace operator for the universal cover. Note that ${\rm H}^1( \Gamma , \ZZ )\ne 0$, because the graph with zero homology is a tree and tree always has a tail. This implies that ${\tilde \Gamma}$ is infinite. Thus the sum in (\ref{deltaformula}) for ${\tilde \Delta}$ is infinite, but ${\tilde z}{\tilde \Delta}$ is a finite element, because ${\tilde z}$ has an expression as a finite sum
\begin{equation}\label{tildez}
{\tilde z}=\sum \lambda_{\gamma}l_{\gamma},
\end{equation}
where $\gamma$ runs over some finite set of contracted paths in ${\tilde \Gamma}$ (with start at ${\tilde t}$),
and $\lambda_{\gamma}\ne 0$.
Let $\gamma_o$ is a contracted path in this sum with maximal length, and $i$ is the ending vertex of this path. Because the vertex is not a tail, there is an edge $l_{ij}$ incident to this vertex, which is not contained in $\gamma_o$. Decomposition of ${\tilde z}{\tilde \Delta}$ into linear combination of contracted paths will contain a non-zero summand $\lambda_{\gamma_o}s_{ij}l_{\gamma_o}l_{ij}$ (no other terms than $\lambda_{\gamma_o}l_{\gamma_o}$ from (\ref{tildez}) can contribute to it). Since ${\widetilde {z\Delta}}={\tilde z}{\tilde \Delta}$, we got a contradiction. The absence of right zero divisors for $\Delta $ is proven similarly.




\end{proof}

It is not hard find graphs with tails for which $\Delta$ is a zero divisor.

Consider the homomorphisms: $\psi_{i}: B(\Gamma) \to \kk \Gamma, i = 1,2$ as in (\ref{hom}).
They map the generators of $B(\Gamma)$ as follows:
\begin{equation}\label{formpsi1}
\psi_1: x_i \mapsto e_i+\sum_j s_{ij}l_{ij},
\end{equation}
\begin{equation}\label{formpsi2}
\psi_2: x_i \mapsto e_i+\sum_j s_{ji}l_{ji}.
\end{equation}

\begin{Corollary}
Let $\Gamma$ be a non-empty connected graph with no tail
and $\kk$ a commutative ring with no zero divisor.
Then homomorphisms $\psi_i: B(\Gamma) \to \kk \Gamma,$ for $i = 1,2$, are injective.
\end{Corollary}

The last remark we would like to make is about 
the definition of Laplace operator. Consider a symmetric square matrix $S=\{s_{ij}\}$ with invertible elements at the diagonal. The combinatorial structure of the graph encodes the places of nonzero entries of the matrix $S$. We define Laplace operator by:
$$
\Delta = \sum s_{ii}e_i+\sum s_{ij}l_{ij}.
$$
The corresponding algebra $\kk \Gamma_{\Delta }$ will have generators $x_i$, corresponding to the trivial 'vertex' paths in the groupoid, which are not projectors anymore. But they can be normalized by a nonzero scalar to be projectors. Multiplication of $x_i$ by $\lambda$ corresponds to multiplication of the elements in $i$-th line and $i$-th row of the matrix $S$ by $\lambda$. These multiplications reduce the problem to the case considered above when all diagonal entries of the matrix are ones. 

\subsection{Properties of Poincare groupoid as a $B(\Gamma )$-module}

We shall prove several useful facts about $\kk\Gamma$ as  $B(\Gamma )$-(bi)module.

\begin{Proposition}\label{bxipi} The isomorphism in theorem \ref{bgammaviagr} identifies
$B(\Gamma )x_i$ with $\ _{\psi_{1}}P_i=\ _{\psi_{1}}(\kk \Gamma e_i)$ (where $\ _{\psi_{1}}P_i$ is $P_i$ with left $B(\Gamma )$-module structure induced by $\psi_1$) and $x_iB(\Gamma )$ with $(e_i\kk\Gamma )_{\psi_{2}}$ (i.e. $e_i\kk\Gamma $ with right $B(\Gamma )$-module structure induced by $\psi_2$).
We have an isomorphism of left $B(\Gamma )$-modules:
\begin{equation}\label{bgxileft}
B^+(\Gamma )=\oplus_i\ B(\Gamma )x_i
\end{equation}
and an isomorphism of right $B(\Gamma )$-modules
\begin{equation}\label{bgxiright}
B^+(\Gamma )=\oplus_i\ x_iB(\Gamma )
\end{equation}
\end{Proposition}
\begin{proof}
Formulas (\ref{xixj0}) and (\ref{xxll}) show that $B(\Gamma )x_i$ coincides with $\kk \Gamma e_i$ and $x_iB(\Gamma )$ does so with $e_i\kk\Gamma $. Since the left $B(\Gamma )$-module structure on $\kk \Gamma$ comes via $\psi_1$ and the right module structure does via $\psi_2$ the identifications follow. The decompositions
$$
\kk\Gamma = \oplus_i \ \kk\Gamma e_i,\ \ \kk\Gamma = \oplus_i \ e_i\kk\Gamma
$$
imply the decompositions for $B^+(\Gamma )$.
\end{proof}

Since $\kk\Gamma$ is isomorphic to $B^+(\Gamma )$, this proposition is applicable to $\kk\Gamma$ when we consider it as a left or right module over $B(\Gamma )$.

\begin{Corollary}\label{bgprojective}
$B^+(\Gamma )$ (hence, $\kk\Gamma$) is projective as left and right $B(\Gamma )$-module.
\end{Corollary}
\begin{proof}
Since, $x_i$ is an idempotent, we have a decomposition:
$$
B(\Gamma )=B(\Gamma )x_i\oplus B(\Gamma )(1-x_i).
$$
Being a direct summand of a free module, module $B(\Gamma )x_i$ is projective. Hence $B^+(\Gamma )$ is so.
\end{proof}

The multiplication in $ A_{\Delta}$ defines an $A$-bimodule  homomorphism $A\otimes_{\widehat A_{\Delta}} A\to A$.
\begin{Proposition}\label{tensormult}
Multiplication map in $B^+(\Gamma )$ defines an isomorphism of $\kk\Gamma$-bimodules:
\begin{equation}\label{tensorkg}
\kk\Gamma\otimes_{B(\Gamma )}\kk\Gamma=\kk\Gamma .
\end{equation}
\end{Proposition}
\begin{proof} $\kk\Gamma$ is identified with $B^+(\Gamma )$ as a left and right $B(\Gamma )$-module. Thus we need to show that $B^+(\Gamma)\otimes_{B(\Gamma )}B^+(\Gamma )=B^+(\Gamma )$. Every element in $B^+(\Gamma )$ has the form
$$
b=\sum x_ib_i,
$$
for some $b_i\in B^+(\Gamma )$, in particular, for $b=x_i$, we have $x_i=x_ix_i$. This implies that $b$ is the image of $x_i\otimes b_i$ under multiplication map, i.e. the map is surjective. Element $p\in B^+(\Gamma )\otimes_{B(\Gamma )}B^+(\Gamma )$ can be presented in the form
$$
p=\sum x_i\otimes c_i.
$$
If $p$ is in the kernel of the multiplication morphism, then $\sum x_ic_i=0$.
By proposition \ref{bxipi}, this implies that $x_ic_i=0$ for all $i$. Then, $x_i\otimes c_i=x_i^2\otimes c_i=x_i\otimes x_ic_i=0$, i.e. $p=0$. Hence, multiplication map is also injective.
\end{proof}

\section{Well-tempered elements and homotopes}
Given an algebra $A$ over a commutative ring $\kk$ and a fixed element $\Delta$, we formalize suitable conditions on this element which imply properties of $B(\Gamma )$ described in the previous section. We call elements satisfying these conditions well-tempered. They provide a good framework for studying representation theory of a class of algebras whom $B(\Gamma )$ belongs to.

\subsection{Projective augmentation ideals}
Here we give a criterion for the augmentation ideal to be a projective right or left module in terms of Noncommutative Differential Geometry. In the definition of well-tempered elements given in the next subsection, the projectiveness of the augmentation ideal is the crucial property.
When considering algebras over a commutative ring $\kk$, we assume them to be free as modules over $\kk$. By default, tensor products are assumed to be over $\kk$.

Recall some facts from Noncommutative Differential Geometry.
Let $B$ be a unital algebra over $\kk$. Then $B$-bimodule $\Omega^1_B$ of noncommutative 1-forms is defined as the kernel of the multiplication map: $B\otimes B\to B$. Thus, we have a short exact sequence:
\begin{equation}\label{omegaoneseq}
0\to \Omega^1_B\to B\otimes B\to B\to 0.
\end{equation}
The universal derivation $B\to \Omega^1_B$ is defined by the formula:
$$
{\rm d}b=b\otimes 1-1\otimes b ,
$$
for $b\in B$.

Let $M$ be a left module over algebra $B$. We say that $\nabla :M\to \Omega^1_B\otimes_B M$ is a connection on $M$ if it satisfies the condition:
$$
\nabla (bm)=({\rm d}b)m+b\nabla (m),
$$
for all $b\in B$ and $m\in M$.
\begin{Lemma} \cite{CQ1}
Module $M$ is projective if and only if there exists a connection upon it.
\end{Lemma}
\begin{proof}

Taking tensor product over $B$ of sequence (\ref{omegaoneseq}) with $M$ gives a short exact sequence:
$$
0\to \Omega^1_B\otimes_BM\to B\otimes M\to M\to 0,
$$
because ${\rm Tor}_1^B(B, M)=0$. Since $B\otimes M$ is a free $B$-module, projectiveness of $M$ is equivalent to splitting of this sequence, i.e. existence of a $B$-module homomorphism $B\otimes M\to \Omega^1_B\otimes_BM$ which is a retraction on $\Omega^1_B\otimes_BM$. When composed with the embedding $M\to B\otimes M$, where $m\mapsto 1\otimes m$, this retraction is nothing but a connection on $M$.
\end{proof}

Note that $\Omega^1_B$ is projective as both left and right $B$-module. The splitting of the sequence (\ref{omegaoneseq}) as of left (respectively, right) $B$-modules is given by map $B\to B\otimes B$ defined by $b\mapsto b\otimes 1_B$ (respectively, $b\mapsto 1_B\otimes b$).

Let $B$ be any augmented unital algebra with augmentation ideal $B^+$.
We shall use the short exact sequence given by the augmentation:
\begin{equation}\label{aug1}
0\to B^+\to B\to \kk \to 0.
\end{equation}

\begin{Lemma}\label{omegaone} We have an isomorphism of left $B$-modules:
$$
\Omega^1_B=B\otimes B^+,
$$
and an isomorphism of right $B$-modules:
$$
\Omega^1_B=B^+\otimes B,
$$
\end{Lemma}
\begin{proof}
By definition of $\Omega^1_B$, we have an embedding $\Omega^1_B\to B\otimes B$. Composite of this map with the projection $B\otimes B\to B\otimes B^+$ that kills $B\otimes 1_B$ gives the required isomorphism for left $B$-modules. The projection $B\otimes B\to B^+\otimes B$ that kills $1_B\otimes B$ does the case for right $B$-modules.
\end{proof}

Now we consider $M=B^+$ as a left or right $B$-module.
\begin{Proposition} \label{projectiveness} Assume that multiplication map $B^+\otimes B^+\to B^+$ is epimorphic. Then the following are equivalent:
\begin{itemize}
\item[(i)]$B^+$ is projective as a left (respectively, right) $B$-module,
\item[(ii)]there is a left (respectively, right) $B$-module homomorphism which is a section to the map
$B^+\otimes B^+\to B^+$ given by multiplication in $B^+$,
\item[(iii)] left $B$-submodule $\Omega^1_BB^+$ in $\Omega^1_B$ is a direct summand  (respectively, right $B$-submodule $B^+\Omega^1_B$ in $\Omega^1_B$ is a direct summand).
\end{itemize}
\end{Proposition}
\begin{proof}
Left projectiveness of $B^+$ is equivalent to the existence of a left $B$-module homomorphism $s:B^+\to B\otimes B^+$, a section to the multiplication map $B\otimes B^+\to B^+$.  By assumption, every element $b\in B^+$ has a decomposition $b=\sum b_ic_i$ with $b_i,c_i\in B^+$. Since $s$ is a $B$-module homomorphism, we got that $s(b)=\sum b_is(c_i)\in B^+\otimes B^+$, i.e. the image of $s$ lies in $B^+\otimes B^+$. Thus, we can consider $s$ as a section required in (ii). This proves equivalence of (i) and (ii).


Apply functor $\Omega^1_B\otimes_B (-)$ to the augmentation sequence (\ref{aug1}). Since $\Omega^1_B$ is projective as a right $B$-module, it is also flat. Thus we get a short exact sequence:
$$
0\to \Omega^1_B\otimes_B B^+\to \Omega^1_B\to \Omega^1_B\otimes_B \kk \to 0.
$$
The image of the first homomorphism is $\Omega^1_BB^+$. In view of lemma \ref{omegaone}, we have isomorphisms
of left $B$-modules $\Omega^1_B=B\otimes B^+$.
Applying functor $(-)\otimes_B \kk$ to (\ref{omegaoneseq}), we get that $\Omega^1_B\otimes_B \kk = B^+$
as a left $B$-module.
Easy calculation shows that, under this identifications,
the homomorphism $\Omega^1_B\to \Omega^1_B\otimes_B\kk$ in the above short exact sequence coincides up to sign with the morphism $B\otimes B^+\to B^+$ given by multiplication. Thus the kernel of the multiplication map $B\otimes B^+\to B^+$ is isomorphic to $\Omega^1_BB^+$. This map has a section if and only if $B^+$ is a projective left $B$-module. Therefore $\Omega^1_BB^+$ is a direct summand in $\Omega^1_B$ exactly when $B^+$ is projective.
\end{proof}

\subsection{Well-tempered elements}

Given an element $\Delta$ in an algebra $A$, let $B$ be the homotope $B={\widehat A_{\Delta}}$ and $B^+=A_{\Delta}$ the augmentation ideal in $B$.

Consider  complex $K_{\Delta}$:
\begin{equation}\label{complexk}
\dots \to A^{\otimes n}\to \dots \to A^{\otimes 3}\to A^{\otimes 2}\to A\to 0
\end{equation}
with differential ${\rm d}:A^{\otimes n+1}\to A^{\otimes n}$ defined by:
$$
d(a_0\otimes \dots \otimes a_{n})=\sum (-1)^ia_0\otimes \dots \otimes a_i\Delta a_{i+1}\otimes \dots \otimes a_n.
$$
This is the bar-complex for $B^+$ expressed in terms of $A$. It might not be exact in general, because $B^+$ is not a unital algebra.

\begin{Lemma}\label{doublecoset}
If $\Delta '=c\Delta d$,
where $c$ and $d$ are any invertible elements in $A$, then the homotopes $B$ and $B'$ constructed by means of $\Delta$ and $\Delta '$ are isomorphic.
\end{Lemma}
\begin{proof}
The isomorphism between $B$ and $B'$ is given by taking the unit to the unit, while the map $B^+\to B'^+$ is defined by
$$
b\mapsto d^{-1}bc^{-1},
$$
where multiplication is taken in $A$.
\end{proof}


{\bf Definition.} An element $\Delta$ in algebra $A$ is said to be {\em well-tempered} if the multiplication map in the homotope $B^+\otimes B^+\to B^+$ is an epimorphism (i.e. $A\Delta A=A$) and $B^+$ is projective as left and right $B$-module.

It is proven in \cite{Zh} that if $A$ is a finite dimensional algebra, then the epimorphism of multiplication implies that $B^+$ is projective, i.e. the second condition on well-tempered elements is redundant. 

\begin{Lemma}
Laplace operator as in (\ref{deltaformula}) is well-tempered element of $\kk\Gamma$.
\end{Lemma}
\begin{proof}
This follows from corollary \ref{bgprojective} and proposition \ref{tensormult}.
\end{proof}


\begin{Lemma} \label{comtemp}
If algebra $A$ is commutative, then $\Delta$ is well-tempered if and only if it is invertible.
\end{Lemma}
\begin{proof}
Note that the image of the multiplication map $B^+\otimes B^+\to B^+$, in terms of $A$, is the two-sided ideal $A\Delta A$. Therefore, if $A$ is  commutative and $\Delta$ is well-tempered, then $\Delta A=A$, i.e. $\Delta$ is invertible. Conversely, if $\Delta$ is invertible, then the multiplication map is epimorphic, and we can define the left (respectively, right) $B$-module section to it by mapping $b\mapsto b\otimes {\Delta}^{-1}$ (respectively, ${\Delta}^{-1}\otimes b$). Proposition \ref{projectiveness} gives the result.
\end{proof}

As we have seen in 5.1, this implies in the commutative case that there is an isomorphism of algebras $B\simeq A\oplus \kk$ .

According to proposition \ref{projectiveness}, $\Delta$ is well-tempered if and only if the most right differential $A\otimes A\to A$ in the complex $K_{\Delta }$ is epimorphic and allows sections which are left and right $B$-module homomorphisms, respectively.

\begin{Lemma} Let $c$ and $d$ be any invertible elements in $A$. Element $\Delta$ is well-tempered if and only if $\Delta '=c\Delta d$ is so.
\end{Lemma}
\begin{proof}
Since the isomorphism in lemma \ref{doublecoset} preserves $B^+$, the property for $B^+$ to be projective is preserved under the isomorphism.
\end{proof}




\begin{Proposition}
If $\Delta$ is well-tempered, then complex $K_{\Delta}$ is exact.
\end{Proposition}
\begin{proof} As we already mentioned, complex $K_{\Delta}$ coincides with the bar-complex for $B^+$:
$$
\dots \to  (B^+)^{\otimes i}\to \dots \to B^+\otimes B^+\to B^+.
$$
Let $h:B^+\to B^+\otimes B^+$ be a homomorphism of right $B$-modules and a section to the most right differential in the complex. It exists by proposition \ref{projectiveness}. Consider the homotopy in this complex defined by :
$$
h(b_1\otimes b_2\otimes b_2\otimes \dots \otimes b_n)=h(b_1)\otimes b_2\otimes \dots \otimes b_n.
$$
One can easily check that it satisfies the equation:
$$
{\rm d}h+h{\rm d}={\rm id}
$$
Hence the complex is exact.
\end{proof}
Note that in the proof of proposition 21 we used only one-sided projectiveness of $B^+$. 
\begin{Corollary}
If $\Delta$ is well-tempered, then multiplication in $B^+$ gives an isomorphism of $A$-bimodules:
\begin{equation}\label{bbbb}
B^+\otimes _BB^+=B^+.
\end{equation}
\end{Corollary}
\begin{proof}
Indeed, the quotient of $(B^+)^{\otimes 2}$ by the image of the differential $(B^+)^{\otimes 3}\to (B^+)^{\otimes 2}$ is isomorphic to $B^+\otimes _BB^+$. Since  there is no homology of $K_{\Delta}$ in the terms $B^+\otimes B^+$ and $B^+$, then this quotient is identified with $B^+$.
\end{proof}

Consider the homomorphism of $A$-bimodules
\begin{equation}\label{b+hom}
A \to {\rm Hom}_B(A , A )
\end{equation}
that takes $a\in A $ to the homomorphism $A \to A$ of left $B$-modules defined by right multiplication with $a$ in $A$. 
\begin{Proposition} Let $\Delta\in A$ be well-tempered, and $B={\widehat A_{\Delta}}$. Then map (\ref{b+hom}) defines an isomorphism of algebras and of $A$-bimodules:
\begin{equation}\label{homkg}
A^{opp} ={\rm Hom}_{B}(A , A )
\end{equation}
\end{Proposition}
\begin{proof} Recall that $B^+=A$ as an $A$-bimodule.
It is a straightforward check that (\ref{b+hom}) defines a homomorphism of algebras and of $A$-bimodules.
Apply functor ${\rm Hom}( -, A )$ to isomorphism (\ref{bbbb}). It gives:
$$
A={\rm Hom}_{A }(A , A )={\rm Hom}_{A }(A\otimes_{B}A , A )={\rm Hom}_{B}(A , A).
$$
Again, it is straightforward to check that this isomorphism coincides with the one we consider.


\end{proof}

Since Laplace operator is a well-tempered element in $\kk \Gamma$, then proposition 23  implies
\begin{Corollary}
We have an isomorphism of algebras and of $\kk\Gamma$-bimodules:
\begin{equation}\label{homkg'}
\kk\Gamma^{opp} ={\rm Hom}_{B(\Gamma )}(\kk\Gamma , \kk\Gamma )
\end{equation}
\end{Corollary}
%



There is a morphism of left $A $-modules
\begin{equation}\label{ahom}
A \to {\rm Hom}_{B}(A , B)
\end{equation}
that takes an element $a\in A$ to the composite of the operator of right multiplication by $a$, 
$$
R_a:A \to A =B^+,
$$ 
with embedding $B^+\to B$.

\begin{Proposition}\label{homkgb} Let $\Delta\in A$ be well-tempered, and $B={\widehat A_{\Delta}}$.
Then morphism (\ref{ahom}) is an isomorphism of left $A$-modules:
$$
A = {\rm Hom}_{B}(A , B).
$$
\end{Proposition}
\begin{proof}
Apply functor ${\rm Hom}_{B}(A , -)$ to the short exact sequence (\ref{aug1}).
We obtain:
\begin{equation}
0\to {\rm Hom}_{B}(A ,B^{+}) \to  {\rm Hom}_{B}(A , B) \to {\rm Hom}_{B}(A , \kk )\to
\end{equation}
The first term in this sequence is isomorphic to $A$ by the previous proposition. The third term is zero, because any homomorphism from $A =B^+$ to the trivial $B$-module $\kk$ must annihilate the image of the multiplication map $B^+\otimes B^+\to B^+$, which is epimorphic by the definition of well-tempered elements.

\end{proof}


\begin{Corollary}\label{homkgb'}
We have an isomorphism of left $\kk\Gamma$-modules:
$$
\kk\Gamma = {\rm Hom}_{B(\Gamma )}(\kk\Gamma , B(\Gamma ))
$$
\end{Corollary}

The finite generation for $B^+$ as a $B$-module, that we observed in the case of $B=B(\Gamma )$ holds also in general for algebras $B$ constructed from a well-tempered elements. More precisely, we have
\begin{Lemma}\label{fingen}
Let $B={\widehat A_{\Delta}}$ be such that the multiplication map $B^+\otimes B^+\to B^+$ is epimorphic. Then $B^+$ is finitely generated as a right and left $B$-module.
\end{Lemma}
\begin{proof}
The multiplication map for $B^+$ is epimorphic if two-sided ideal $A\Delta A$ is $A$. Therefore, we have a decomposition for the unit in $A$:
$$
1_A=\sum x_i\Delta y_i.
$$
It implies that any $b\in B^+=A$ has a decomposition:
$$
b=\sum bx_i\Delta y_i=\sum (bx_i)\cdot_By_i,
$$
where $\cdot_B$ is multiplication in $B$. Therefore, $B^+$ is left generated by the finite set $\{y_i\}$. Similarly, it is right generated by the set $\{x_i\}$.
\end{proof}

\subsection{Principal double-sided ideals in matrix algebras and Morita equivalence}

According to proposition \ref{moritakg}, algebra $A=\kk {\Gamma }$ is isomorphic to the matrix algebras over the group ring of the fundamental group of the graph.
The condition for an element $\Delta $ in an algebra to be well-tempered includes the property that the double sided ideal generated by $\Delta$ is the whole algebra $A$. We explore this condition via Morita equivalence in this subsection.

For a unital $\kk$-algebra $C$, consider the matrix algebra 
$$
A={\rm Mat}_{n\times n}(C).
$$
These two algebras are known to be Morita equivalent, which means that the categories of right and left modules over $C$ and $A$ are equivalent. We will need an equivalence between the categories of $C$ and $A$ bimodules. Let $P$ be the $C-A$ bimodule whose elements are vector-rows of size $n$ over $C$ and $Q$ the $A-C$ bimodule of vector-colomns of size $n$ over $C$. The equivalence of the categories of left $A$ and $C$ modules is given by the functor which is defined on the left $A$ module $M$ by $
M\mapsto P\otimes_AM$.
Similarly, the equivalence of the categories of right modules is given by the functor that takes the right $A$ module $M$ to $M\otimes_AQ$. Note that $P\otimes_AQ=C$, which implies that the image of $Q$ under the first functor is $C$, as so is the image of $P$ under the second functor.

The equivalence of the category of $A$ bimodules with the categories of $C$ bimodules is defined on an $A$ bimodule $M$ via the rule:
\begin{equation}\label{moritamatrix}
M\mapsto P\otimes_AM\otimes_AQ.
\end{equation}

\begin{Lemma}
Let $C$ be a unital $\kk$-algebra and $A={\rm Mat}_{n\times n}(C)$ the matrix algebra over $C$. Let $\Delta$ be an element in $A$ with the matrix entries $\Delta_{ij}$. Then the following are equivalent:

i) $A\Delta A=A$,

ii) $C\{ \Delta_{ij}\}C=C$, where the left hand side is the double-sided ideal generated by $\Delta_{ij}$'s.
\end{Lemma}

\begin{proof}
This follows from Morita equivalence. Indeed, $A\Delta A\subset A$ is an embedding of $A$ bimodules. The functor (\ref{moritamatrix}) takes $A$ to $C$ and $A\Delta A$ to the double-sided ideal generated by coefficients of $\Delta$. Since the functor is an equivalence the fact follows.
\end{proof}

There is also a categorical interpretation of the same condition on $\Delta$. If $\Delta \in A={\rm Mat}_{n\times n}(C)$, then it defines a homomorphism of right $C$ modules ${\tilde \Delta} :C^n\to C^n$. Given a left $C$ module $M$, we define $\kk$-homomorphism $\Delta_M : M^n\to M^n$ by taking the tensor product of ${\tilde \Delta}$ with $M$ over $C$.

\begin{Proposition} The following are equivalent:

i) $A\Delta A=A$,

ii) For any left $C$-module $M\ne 0$, homomorphism $ \Delta_M$ is not zero.
\end{Proposition}
\begin{proof}
If $A\Delta A=A$, then we have a decomposition $1=\sum a_i\Delta b_i$, with $a_i,b_i\in A$. Hence, if $\Delta_M =0$, then $1_M =0$, i.e. $M=0$.

Conversely, assume that $\Delta _M\ne 0$ for any $M\ne 0$. Consider $M=C/C\{ \Delta_{ij}\}C$. Clearly, $\Delta_M=0$. Hence $M=0$ and $C\{ \Delta_{ij}\}C=C$. By the above lemma, $A\Delta A=A$.
\end{proof}

\subsection{Hochschild dimension and global dimension of homotopes}

{\em Hochschild dimension over} $\kk$ of a $\kk$-algebra $B$ is defined as projective dimension of $B$ as a $\kk$-central $B$-bimodule, i.e. as a module over $B\otimes B^{opp}$ (note that the tensor product is taken over $\kk$). We denote it by ${\rm Hdim}B={\rm Hdim}_{\kk}B$. It is invariant under Morita equivalences.

In more general approach, when $B$ is a DG-algebra, it is called {\em smooth} if $B$ is a perfect $B\otimes B^{opp}$-bimodule. Smoothness is a derived Morita invariant property, but Hochschild dimension might not be preserved under derived equivalences. For an ordinary algebra $B$, smoothness is equivalent to finiteness of Hochschild dimension.

\begin{Lemma} Let $\Delta $ be a well-tempered element in algebra $A$ and $B={\widehat A_{\Delta}}$. If $P$ is a projective $A$-bimodule, then $_{\psi_1}P_{\psi_2}$ is a projective $B$-bimodule. The similar statement holds for projective left and right $A$-modules.
\end{Lemma}
\begin{proof} We will prove the lemma for $A$-bimodules, the proof for left and right modules is similar. Since $B^+$ is projective as a left and right $B$-module, then $B^+\otimes B^{+opp}$ is projective as a $B\otimes B^{opp}$-module. Projective $A\otimes A^{opp}$-module $P$ is a direct summand in a free $A\otimes A^{opp}$-module $A\otimes U\otimes A^{opp}$, where $U$ is a free $\kk$-module. Therefore, $_{\psi_1}P_{\psi_2}$ is a direct summand in $_{\psi_1}A\otimes U\otimes A^{opp}_{\psi_2}= B^+\otimes U\otimes B^{+opp}$, which is a projective $B\otimes B^{opp}$-module. Therefore, it is projective itself.
\end{proof}
\begin{Theorem} Let $\Delta$ be a well-tempered element in algebra $A$ and $B$ the corresponding homotope $B={\widehat A_{\Delta}}$. Then ${\rm Hdim}B$ is less than or equal to ${\rm max}({\rm Hdim}A, 2)$.
\end{Theorem}
\begin{proof} Since $A$ has a projective resolution of length ${\rm Hdim}A$ as $A\otimes A^{opp}$-module, then $B^+$, which is isomorphic to $_{\psi_1}A_{\psi_2}$, has a projective resolution of the same length as $B\otimes B^{opp}$-module, by the above lemma.

The trivial module $\kk$ has the following resolution of length 2 by projective $B\otimes B^{opp}$-modules:
$$
0\to B^+\otimes B^+\to (B\otimes B^+)\oplus (B^+\otimes B)\to B\otimes B\to \kk \to 0.
$$
The augmentation exact sequence (\ref{aug1}) implies that the projective $B\otimes B^{opp}$-dimension of $B$ is not greater than the maximum of projective dimensions of $B^+$ and of $\kk$.
\end{proof}

Thus, we see that the smoothness of $A$ in the DG-sense implies the smoothness of $B$.
\begin{Corollary}
\label{hochshield dimension of bg}
Hochschild dimension of $B(\Gamma )$ over $\kk$ is less than or equal to 2.
\end{Corollary}
\begin{proof} By proposition \ref{moritakg}, algebra $\kk\Gamma$ is Morita equivalent to a matrix algebra over $\kk[\pi (\Gamma ),t]$. Since the fundamental group of $\Gamma$ is free, algebra $\kk[\pi (\Gamma ),t]$ is quasi-free relatively over $\kk$ in the sense of Cuntz and Quillen \cite{CQ1} (for the relative version of quasi-freeness see \cite{BZh}), hence Hochschild dimension of $\kk\Gamma$ is $\le 1$. As Hochschild dimension is Morita invariant, the above theorem gives the required upper bound for Hochschild dimension of $B(\Gamma )$.
\end{proof}

Recall that the left (respectively, right) global dimension of an algebra is the maximum of projective dimensions of left (respectively, right) modules over the algebra. For an algebra $B$, we denote its left global dimension by ${\rm gldim}_lB$ and right global dimension by ${\rm gldim}_rB$.

As a consequence of theorem 31, we obtain
\begin{Theorem}\label{globaldim}Let $\Delta $ be a well-tempered element in algebra $A$ and $B={\widehat A_{\Delta}}$.
Then we have inequalities for the left and right global dimensions of $B$:
$$
{\rm gldim}_lB\le {\rm max}({\rm Hdim}A, 2)+{\rm gldim}_l\kk,
$$
$$
{\rm gldim}_rB\le {\rm max}({\rm Hdim}A, 2)+{\rm gldim}_r\kk .
$$
\end{Theorem}
\begin{proof}
Let $M$ and $N$ be any two left (or right) $A$-modules. There is the spectral sequence with the sheet ${\rm} E_2$:
\begin{equation}\label{spectralseq}
{\rm E}_2^{ij}={\rm Ext}^i_{A-A}(A, {\rm Ext}_{\kk}^{j}(M, N))
\end{equation}
that converges to ${\rm Ext}^{i+j}_A(M, N)$.
Thus, we get an upper bound for (left or right) global dimension of $B$ from the upper bound on Hochschild dimension obtained in the previous theorem and from the upper bound on $j$ for non-zero ${\rm Ext}^j_{\kk}$ by global dimension of $\kk$.
\end{proof}
\begin{Corollary}\label{globaldim2} If $\kk =k$ is a field, then
the global dimension of the category of $B(\Gamma)$-modules is not higher than 2.
\end{Corollary}

In \cite{Zh}, it was shown in the case of finite dimensional algebras  that if $\Delta$ is not well-tempered, then the left and right dimensions of the homotope are infinite. 

\section{Representation theory for homotopes.}
\label{repbg}

\subsection{General representation theory for homotopes}
\label{genrepbg}
Here we describe functors between categories of modules over an algebra and its homotope. We do not assume in this subsection that $\Delta$
is well-tempered.


Assume again that we have a unital algebra $A$ and a fixed element $\Delta$ in $A$. Let $B^+=A_{\Delta}$ and  $B=\widehat{A}_{\Delta}$ be the non-unital and unital homotopes.


We define push-forward functors on the categories of left modules by restricting the module structure along $\psi_1$ and $\psi_2$,:
$$
\psi_{1*}:  A-{\rm mod} \to  B-{\rm mod},\ \ \psi_{2*}:  A-{\rm mod} \to B-{\rm mod}.
$$
We use notations for these functors that are compatible with the viewpoint of Noncommutative Algebraic Geometry,
where homomorphisms $\psi_1$ and $\psi_2$ are assumed to define geometric maps between the noncommutative affine spectra of
unital algebras: ${\rm Spec}A\to {\rm Spec}B$. The noncommutative affine spectrum of an algebra is understood
as an object of the category opposite to the category of associative algebras. Modules have the meaning of  sheaves on the affine spectra.

There is a natural transformation of functors:
\begin{equation}\label{lowpsi}
\lambda :\psi_{1*} \to \psi_{2*},
\end{equation}
For a representation $\rho :A\to {\rm End}V$ and $v\in V$, it is defined by
\begin{equation}\label{formulaforlambda}
\lambda (v)=\rho (\Delta )v.
\end{equation}
It is a straightforward check that this formula defines a natural transformation.

We say that a $B$-module is $B^+$-{\em trivial}, if $B^+$ acts by zero on it.

\begin{Lemma}\label{trivialconefirst}
Let $W$ be an $A$-module. Then
the kernel and cokernel of $\lambda_W$ are $B^+$-trivial modules.
Moreover, we have an exact sequence with the middle morphism $\lambda_W$:
\begin{equation}\label{psilow12}
0\to (\psi_{1*}W)^{B^+}\to \psi_{1*}W\to \psi_{2*}W\to \psi_{2*}W/(B^+\psi_{2*}W)\to 0.
\end{equation}
\end{Lemma}
\begin{proof} The map $\lambda_W$ is given by the action of $\Delta$ on the representation space. This implies that the action of $\psi_1(b)=b\Delta$ on the kernel of $\lambda_W$ is zero for any $b\in B^+$. By taking $b=1_A$, we see that the kernel is exactly the submodule in $\psi_{1*}W$ which contains all elements on which $B^+$ acts trivially. Hence it is  $(\psi_{1*}W)^{B^+}$.

The image of $\lambda_W$ is the image of the action of $\Delta$. It contains the image of the action of $\psi_2(b)=\Delta b$, for any $b\in B^+$, and, in fact, coincides with $B^+\cdot \psi_{2*}W$ (again, consider $b=1_A$). Thus the quotient of $\psi_{2*}W$ by the image is the space of co-invariants for $B^+$-action.

\end{proof}

Functors $\psi_{1*}$ and $\psi_{2*}$ have right adjoints $\psi^!_1,\ \psi^!_2 : B - {\rm mod} \to A-{\rm mod}$ defined by:
\begin{equation}
\psi^!_1,\ \psi^!_2: V\mapsto {\rm Hom}_{B}(A, V),
\end{equation}
where $A$ is endowed with the structure of left $B$-module via $\psi_1$ or $\psi_2$ respectively.

Also functors $\psi_{1*}$ and $\psi_{2*}$ have left adjoints $\psi^*_1,\ \psi^*_2 : B - {\rm mod} \to A-{\rm mod}$ defined by:
\begin{equation}
\psi^*_1,\ \psi^*_2:  V \mapsto A\otimes_{B}V,
\end{equation}
where $A$ is endowed with right $B$-module structure via $\psi_1$ or $\psi_2$ respectively.


In order to distinguish the multiplication in $A$ from that in $B^+$, we denote it by $\cdot _A$.
Let $\rho : B\to {\rm End}(V)$ be a representation of $B$.
Consider the map
$$
\mu_V: A\otimes_{B} V\to  {\rm Hom}_{B}(A, V)
$$
defined by:
\begin{equation}\label{mu1}
a\otimes v\mapsto \phi_{a\otimes v} \in {\rm Hom}_{B}(A, V),
\end{equation}
where
\begin{equation}\label{phimu}
\phi_{a\otimes v} (a')= \rho (a'\cdot_Aa)v
\end{equation}
Note that we used $\psi_2$ to define $A\otimes_{B} V$ and $\psi_1$ to define ${\rm Hom}_{B}(A, V)$.
\begin{Lemma}
Formulas (\ref{mu1}) and (\ref{phimu}) define a natural transformation of functors:
\begin{equation}\label{naturalmu}
\mu: \psi_2^*\to \psi_1^!.
\end{equation}
\end{Lemma}
\begin{proof} There are several things that we need to check. First, morphism $\phi$ is indeed a homomorphism of left $B$-modules. Second,  morphism $\mu_V$ is well-defined on $A\otimes_{B}V$. Third, morphism $\mu_V$ is compatible with  left $A$-module structures. Forth, morphisms $\mu_V$ are functorial with respect to $V$. All this is a straightforward check, which we leave to the reader.
\end{proof}
By adjunction, we have a natural transformation:
$$
\chi : \psi_{1*}\psi_2^*\to {\rm id}.
$$
Let $V$ be a $B$-module with the action of $B$ given by $\rho :B\to {\rm End}V$. Then $\psi_{1*}\psi_2^*V=A\otimes V$ and
$$
\chi(a\otimes v)=\rho (a)v,
$$
where $a$ is interpreted as an element in $B^+$.

Denote the following adjunction morphisms by $\epsilon$ and $\delta$:
$$
\epsilon: \psi_{1*}\psi_1^!\to {\rm id},
$$
$$
\delta: {\rm id}\to \psi_{2*}\psi_2^*.
$$
For a $B$-module $V$ and $\varphi \in {\rm Hom}_B(A, V)=\psi_{1*}\psi_1^!V$, we have
$$
\epsilon_V(\varphi )=\varphi(1).
$$
For $v\in V$, we have
$$
\delta_V(v)= 1\otimes v
$$
as an element in $A\otimes_BV=\psi_{2*}\psi_2^*V$.
\begin{Lemma}\label{epsilondelta}
The kernel and cokernel of both $\epsilon$ and $\delta$ evaluated on any $B$-module $V$ are $B^+$-trivial modules.
\end{Lemma}
\begin{proof} We denote by $\rho$ the action of $B$ on $V$.
Let $\varphi \in {\rm Hom}_B(A, V)=\psi_{1*}\psi_1^!V$. If $\varphi$ belongs to the kernel of $\epsilon_V$, then $\varphi (1)=0$ and
$$
(b\cdot \varphi ) (a)=\varphi (a\cdot_Ab\cdot_A\Delta )=\rho (a\cdot_A b)\varphi (1)=0,
$$
for any $a\in A$ and $b\in B^+$. Hence the kernel of $\epsilon$ is a $B^+$-trivial $B$-module.

If $v=\rho (b)u$ for some $b\in B^+$ and $u\in V$, then define $\varphi \in {\rm Hom}_B(A, V)$ by
$\varphi (a)=\rho (a\cdot_A b)u$. This is indeed a homomorphism of $B$-modules and $\varphi (1)=v$. Hence the action by elements in $B^+$ on $V$ has values in the image of $\epsilon$, i.e. the cokernel of $\epsilon$ is a $B^+$-trivial $B$-module.

Let $a\otimes v$ be in $A\otimes_BV=\psi_{2*}\psi_2^*V$ and $b\in B^+$. Then
$$
b\cdot (a\otimes v)=(\Delta \cdot_Ab\cdot_Aa)\otimes v=1\otimes \rho(b\cdot_Aa)v.
$$
Thus, the cokernel of $\delta_V$ is a $B^+$-trivial $B$-module.

Let $v\in V$ is such that $\delta_V(v)=1\otimes v=0$ in $A\otimes_BV=\psi_{2*}\psi_2^*V$. Using the fact that $A\otimes_BV$ is a left $A$-module, we get $a\otimes v=a(1\otimes v)=0$, for any $a\in A$. Now let us use the fact that $A\otimes_BV$ is the underlying vector space for $\psi_{1*}\psi_2^*V$, though with a different $B$-module structure, and apply $\chi_V :A\otimes_BV\to V$. We get:
$$
\chi_V(a\otimes v)=\rho (a)v=0,
$$
where $a$ is now interpreted as an arbitrary element in $B^+$. Therefore, the kernel of $\delta_V$ is a $B^+$-trivial $B$-module.
\end{proof}

\begin{Lemma}\label{3composite}Let $V$ be a $B$-module.
The composite of natural transformations:
\begin{equation}
\xymatrix{\psi_{1*}\psi_1^!V\ar[r]^{\epsilon_V} & V \ar[r]^{\delta_V} &  \psi_{2*}\psi_2^*V\ar[r]^{\psi_{2*}(\mu_V)} & \psi_{2*}\psi_1^!V}
\end{equation}
coincides with $\lambda_{\psi_1^!V}$.
\end{Lemma}
\begin{proof} This is a straightforward check.
\end{proof}

{\bf Example.} If $A$ is a commutative algebra, then $\psi_1=\psi_2$. Let $A=k[t]$ be the algebra of polynomial in one variable and $\Delta =t^2$. Algebra $B={\widehat A_{\Delta}}$ is the algebra of an affine curve with a cusp:
$$
B=k[x, y]/(x^3-y^2).
$$
The homomorphism $\psi_1=\psi_2:B\to A$ is the normalization map for the cusp curve given by:
$$
x=t^2,\ \ y=t^3.
$$
Take $V=A$ as a $B$-module. Then an easy calculation shows that
$$
A\otimes_BA\simeq k[t]\oplus k[t]/t^2
$$
as a left $A$-module, and
$$
{\rm Hom}_B(A, A)\simeq k[t].
$$
Thus $\mu_A$ is not an isomorphism. Note that $\Delta$ is not well-tempered in this example in view of lemma \ref{comtemp}. We will see that in the well-tempered case $\mu$ is an isomorphism of functors.

\subsection{Calculations in the well-tempered case}

\begin{Lemma}\label{exactfunctors}
Let $\Delta$ be a well-tempered element in $A$. Then functors $\psi^!_1$ and $\psi^*_2$ are exact.
\end{Lemma}
\begin{proof}
Since $B^+$ is projective as a left $B$-module, then $\psi^!_1$ is exact. Since $B^+$ is projective, hence flat, as a right $B$-module, then $\psi^*_2$ is exact.
\end{proof}
\begin{Lemma}\label{zerohoms}
Let $\Delta$ be well-tempered, $W$ an $A$-module and $V$ a $B^+$-trivial $B$-module.
Then
\begin{itemize}
\item[(i)] $\RR \psi_1^!V={\mathbb R}{\rm Hom}_{B}(A , V)=0;$
\item[(ii)] ${\mathbb L} \psi_2^*V=A \otimes_{B}^{\mathbb L} V=0;$
\item[(iii)] ${\rm Ext}_{B}^{\bullet}(\psi_{1*}W, V)=0;$
\item[(iv)]${\rm Ext}_{B}^{\bullet}(V, \psi_{2*}W)=0.$
\end{itemize}
\end{Lemma}
\begin{proof} First,
$$
{\mathbb R}{\rm Hom}_{B}(A , V)=\RR {\rm Hom}_{B}(B^+, V)={\rm Hom}_{B}(B^+, V),
$$
because $B^+$ is a projective left $B$-module.
Further any homomorphism from $A =B^+$ to a $B^+$-trivial $B$-module must annihilate the image
of the multiplication map $B^+\otimes B^+\to B^+$. This map is epimorphic by the definition of well-tempered elements.
This implies (i).

We have by adjunction:
$$
{\rm Ext}_{B}^{\bullet}(\psi_{1*}W, V)={\rm Ext}_{A}^{\bullet}(W,\RR \psi_1^!V)=0,
$$
which proves (iii).

Similarly,
$$
A \otimes_{B}^{\mathbb L} V=B^+\otimes_{B}^{\mathbb L} V=B^+\otimes_{B}V
=0,
$$
because every element $b\in B^+$ has the form $b=\sum a_i\cdot_Bb_i$, where $a_i,b_i\in B^+$. Hence $b\otimes v=\sum a_i\otimes b_iv=0$ for any $v\in V$. This proves (ii) and implies
$$
{\rm Ext}_{B}^{\bullet}(V, \psi_{2*}W)={\rm Ext}_{A}^{\bullet}({\mathbb L} \psi_2^*V, W)={\rm Ext}_{A}^{\bullet}(A \otimes_{B}^{\mathbb L} V, W)=0,
$$
which proves (iv).
\end{proof}
\begin{Proposition}\label{adjunction}
The adjunction morphisms
$$
\psi_2^*\psi_{2*}\to {\rm id},\ \ {\rm id}\to \psi_1^!\psi_{1*}
$$
are isomorphisms in the well-tempered case.
\end{Proposition}
\begin{proof}
Let $W$ be a left $A$-module. Since $A$ is isomorphic $B^+$, isomorphism (\ref{bbbb}) reads as $A\otimes_BA=A$. Applying tensor product $(-)\otimes_{A}W$ to it gives:
$A\otimes_{B}W=W$, which implies that the morphism $\eta_W:\psi_2^*\psi_{1*}W\to W$ obtained from $\lambda_W$ by adjunction is an isomorphism. Now apply $\psi_2^*$ to $\lambda_W$.
By lemmas \ref{trivialconefirst}, \ref{exactfunctors} and \ref{zerohoms}(ii), we get that
$\psi_2^*(\lambda_W):\psi_2^*\psi_{1*}W\to \psi_2^*\psi_{2*}W$ is an isomorphism.
Since $\eta_W$ is the composite of $\psi_2^*(\lambda_W)$ and the adjunction morphism $\psi_2^*\psi_{2*}W\to W$, we get that $\psi_2^*\psi_{2*}\to {\rm id}$ is an isomorphism.

Now, applying tensor product $(-)\otimes_{A}W$ to isomorphism (\ref{homkg}) implies that ${\rm id}\to \psi_1^!\psi_{1*}$ is an isomorphism.
\end{proof}


\begin{Proposition}\label{naturaliso}
Let $\Delta$ be a well-tempered element in algebra $A$ and $B={\widehat A_{\Delta}}$. Then natural transformation $\mu$ from (\ref{naturalmu}) gives an isomorphism $\psi_2^*\simeq \psi_1^!$ on the category $B-{\rm mod}$.
\end{Proposition}
\begin{proof} Functor $\psi_2^*$ takes rank 1 free $B$-module $B$ to $A$ and so does the functor $\psi_1^!$ due to proposition \ref{homkgb}. Functors $\psi_2^*$ and $\psi_1^!$ are exact and commute with infinite direct sum. The latter does, because $A$ is a finitely generated left $B$-module due to lemma \ref{fingen}. Hence, a presentation for a $B$-module $V$ as a cokernel of a homomorphism of free $B$-modules implies similar presentation for $\psi_2^*V$ and $\psi_1^!V$ as cokernels of a homomorphisms of free $A$-modules, while $\mu$ induces an isomorphism of these presentations.
\end{proof}



\begin{Proposition}\label{exceptional} Let $\Delta$ be a well-tempered. For any two $B^+$-trivial $B$-modules $U$ and $V$ and any $i\in \ZZ$, we have:
$$
{\rm Ext}^i_B(U, V)={\rm Ext}^i_{\kk}(U, V).
$$
In particular, $\kk$ is an exceptional object, i.e. ${\rm Hom}_{B}(\kk, \kk)=\kk$ and
${\rm Ext}^i_{B}(\kk, \kk)=0$, for $i\ne 0$.
\end{Proposition}
\begin{proof} Clearly, ${\rm Hom}_{B}(U, V)={\rm Hom}_{\kk}(U, V)$. Let $F$ be a free $\kk$-module. Let us show that ${\rm Ext}^{>0}_B(F, V)=0$, for any $B^+$-module $V$. Applying functor $(-)\otimes F$ to the augmentation exact sequence (\ref{aug1}) gives a short exact sequence of $B$-modules:
$$
0\to B^+\otimes F\to B\otimes F\to F\to 0
$$
Apply functor ${\rm Ext}^i_B(-, V)$ to it. Since $B^+$ and $B$ are projective $B$-modules, we get ${\rm Ext}^{>1}(F, V)=0$ and an exact sequence:
$$
0\to {\rm Hom}_B(F, V)\to {\rm Hom}_B(B\otimes F, V)\to {\rm Hom}_B(B^+\otimes F, V)\to {\rm Ext}^1_B(F, V)\to 0
$$
Since $B^+=\psi_{1*}A$, then ${\rm Hom}_B(B^+, V)=0$ by lemma \ref{zerohoms} (iii).
Since $F$ is free, we have: ${\rm Hom}_B(B^+\otimes F, V)=0$ and, in view of the exact sequence, ${\rm Ext}^1_B(F, V)=0$, too. The calculation of ${\rm Ext}^i_B(U, V)$ via free $\kk$-resolutions implies the required isomorphism wit ${\rm Ext}^i_{\kk}(U, V)$.

\end{proof}
\begin{Proposition} Let $\Delta$ be well-tempered and $W$ an $A$-module. Then we have an isomorphism of $B$-modules:
\begin{equation}\label{hompsi12}
{\rm Hom}_B(B^+, \psi_{1*}W)=\psi_{2*}W.
\end{equation}
Further, there is a quasi-isomorphism of complexes:
\begin{equation}\label{hominpsi1}
\RR {\rm Hom}_B(\kk , \psi_{1*}W)=\{
\xymatrix{
0\ar[r] & \psi_{1*}W \ar[r]^{\lambda_W} & \psi_{2*}W\ar[r] &  0}\},
\end{equation}
with the latter having two non-trivial components in degrees 0 and 1.
\end{Proposition}
\begin{proof}
We have an isomorphism of $A$-modules ${\rm Hom}_B(A, \psi_{1*}W)=\psi_1^!\psi_{1*}W=W$ by proposition \ref{adjunction}. ${\rm Hom}_B(B^+, \psi_{1*}W)$ is the same module with the left action of $B$ coming from its right action on $A$,  defined by $\psi_2$. This proves (\ref{hompsi12}).

By using augmentation sequence (\ref{aug1}) as a projective  resolution for left $B$-module $\kk$ and the above identification for ${\rm Hom}_B(B^+, \psi_{1*}W)$, we obtain (\ref{hominpsi1}).
\end{proof}

\begin{Lemma}\label{homkb+} Let $\Delta$ be well-tempered and $W$ an $A$-module. Then:
${\rm Hom}_B(\kk , \psi_{1*}W)=(\psi_{1*}W )^{B^+}$, ${\rm Ext}^1_B(\kk , \psi_{1*}W)= \psi_{2*}W/B^+\psi_{2*}W$,
and ${\rm Ext}^{> 1}_B(\kk , \psi_{1*}W)=0$. In particular, ${\rm Hom}_B(\kk , B^+)=\{a\in A| \Delta a=0\}$, ${\rm Ext}^1_B(\kk , B^+)= A/\Delta A$ and ${\rm Ext}^{> 1}_B(\kk , B^+)=0$.
\end{Lemma}
\begin{proof} This follows from (\ref{hominpsi1}) and (\ref{psilow12}). The particular case is when we take $W=A$.
\end{proof}



\begin{Proposition}\label{globaldim1} Let $\Delta$ be well-tempered, $V$ a $B$-module and $U$ a $B^+$-trivial $B$-module.
Then ${\rm Ext}_B^i(V,U)=0$, for $i> {\rm gldim}\kk+1$.
\end{Proposition}
\begin{proof}
By lemma \ref{epsilondelta} we have an exact sequence
\begin{equation}
\label{exactsequence}
0\to K\to \psi_{1*}\psi_1^!V\to V\to Q\to 0
\end{equation}
with $K$ and $Q$ $B^+$-trivial $B$-modules.
Consider the spectral sequence that calculates functor ${\rm Ext}_{B}^{\bullet} (-, U)$ on this exact sequence. Its $E_1$-term has Ext-groups on top of terms of the sequence and it converges to zero, because the sequence is exact. We have by adjunction:
$$
{\rm Ext}^{\bullet}_{B}(\psi_{1*}\psi_1^!V, U)={\rm Ext}_{A}^{\bullet}(\psi_1^!V, \psi_1^!U),
$$
because $\psi_1^!$ is an exact functor by lemma \ref{exactfunctors}.
By lemma \ref{zerohoms}, $\psi_1^!U=0$. Therefore, ${\rm Ext}^{\bullet}_{B}(\psi_{1*}\psi_1^!V, U)=0$.
Moreover, since $K$ and $Q$ are trivial modules, ${\rm Ext}_{B}^{i}(K, U)=0$ and ${\rm Ext}_{B}^{i}(Q, U)=0$, for $i> {\rm gldim}\kk$, by proposition \ref{exceptional}.
The spectral sequence immediately implies that ${\rm Ext}_{B}^{i}(V, U)=0$, for $i> {\rm gldim}\kk+1$.
\end{proof}

\subsection{Coherence of algebras and categories of finitely presented modules}

Let $A=\kk \Gamma$. Then we know that $B(\Gamma )=\widehat{A}_{\Delta}$. According to our original problem, we are interested in the categories of  representations, which are finitely generated over $\kk$. A counterpart of the category of coherent sheaves on an affine algebraic variety in Noncommutative Algebraic Geometry is the category of {\em finitely presented} modules, while finite over $\kk$ $A$-modules are only analogous to Artinian sheaves, i.e. sheaves supported at points.

In order to work with the category of finitely presented modules, we need first to show that it is abelian.

A left module $M$ over a ring is said to be {\em coherent} if it is finitely generated and for every morphism $\varphi : P\to M$ with free module $P$ of finite rank the kernel of $\varphi$ is finitely generated too. A ring is {\em (left) coherent} if it is coherent as a left module over itself. If a ring is coherent, then finitely presented modules are the same as coherent modules and the category of finitely presented modules is abelian. Thus we need to check when the homotope is coherent.

By proposition \ref{moritakg}, algebra $\kk\Gamma$ is isomorphic to the matrix algebra over $\kk[\pi (\Gamma , t)]$. Hence, it is quasi-free relatively over $\kk$. The definition  of a quasi-free algebra in the relative case (i.e. for algebras over a commutative ring $\kk$ rather than over a field) is a straightforward generalization of the one due to Cuntz and Quillen \cite{CQ1} (see \cite{BZh}). Algebra $A$ is {\em quasi-free} if the module of noncommutative differential 1-forms
$\Omega^1A$
is a projective $A\otimes A^{opp}$-module.

A direct generalization of Cuntz-Quillen criterion  for quasi-freeness \cite{CQ1} holds in the relative case. Recall that a {\em square zero nilpotent extension} of an algebra $R$ is an algebra ${\tilde R}$ together with a ring epimorphism ${\tilde R}\to R$, whose kernel $I$ satisfies $I^2=0$.
\begin{Proposition}\cite{BZh}
\label{qfcriterion}
Algebra $A$ is quasi-free over $\kk$ if and only if for any ${\tilde R}$, a square zero nilpotent extension of the algebra $R$ in the category of $\kk$-algebras, and any homomorphism $A\to R$, there exists its lifting to a homomorphism $A\to {\tilde R}$.
\end{Proposition}
The following criterion of coherence for quasi-free algebras is proven in \cite{BZh} with the help of Chase criterion \cite{Chase}.
\begin{Theorem}\label{thquasifree}
Let $\kk$ be a commutative noetherian ring. Assume that $A$ is an algebra quasi-free relatively over $\kk$ and $A$ is flat as a $\kk$-module. Then $A$ is a left and right coherent algebra.
\end{Theorem}
Now we present a theorem that allows us to iteratively construct coherent algebras via homotopes.

\begin{Theorem}\label{bcoherent}
Let $\Delta \in A$ be well-tempered and $B={\widehat A_{\Delta}}$. Assume that $A$ is a coherent algebra over a noetherian commutative ring $\kk$. Then $B$ is coherent too.
\end{Theorem}
\begin{proof}
Consider a homomorphism $M\to N$ of free $B$-modules of finite rank. Embed it into a four term exact sequence
\begin{equation}\label{kmnq}
0\to K\to M\to N\to Q\to 0.
\end{equation}
We need to show that $K$ is finitely generated. Apply functor $\psi_{1*}\psi_1^!$ to sequence (\ref{kmnq}. Since both functors $\psi_{1*}$ and $\psi_1^!$ are exact (lemma \ref{exactfunctors}), we get an exact sequence:
$$
0\to \psi_{1*}\psi_1^!K\to \psi_{1*}\psi_1^!M\to \psi_{1*}\psi_1^!N\to \psi_{1*}\psi_1^!Q\to 0
$$
By proposition \ref{homkgb}, $\psi_1^!M$ and $\psi_1^!N$ are free $A$-modules of finite rank. Since $A$ is a coherent algebra, this implies that $\psi_1^!K$ is a finitely generated $A$-module. Further, $\psi_{1*}A$ is finitely generated as a $B$-module by lemma \ref{fingen}. Therefore, $\psi_{1*}$ of a finitely generated $A$-module is a finitely generated $B$-module. Hence, $B$-module $\psi_{1*}\psi_1^!K$ is finitely generated too.

Now consider the exact sequence induced by the adjunction morphism:
$$
\psi_{1*}\psi_1^!K\to K\to S\to 0
$$
As $\psi_{1*}\psi_1^!K$ is finitely generated, it is remained to show that $S$ is so. By lemma \ref{epsilondelta}, $S$ is a $B^+$-trivial $B$-module. Therefore, $S$ is a quotient of $K/ B^+K$. Let us check that the latter is finitely generated.

Consider the spectral sequence which is obtained by applying functor ${\rm Tor}^B_{\bullet}(\kk, -)$ to the sequence (\ref{kmnq}). The spectral sequence converges to zero, because the original sequence is exact. Taking into account that $M$ and $N$ are projective, hence, flat, we see that the contribution to $\kk\otimes_BK=K/B^+K$
come from a submodule of $\kk\otimes_B M$ and ${\rm Tor}_2^{B}(\kk , Q)$.
Module $\kk\otimes_B M$ is a finite rank free $\kk$-module. In view of noetherianess of $\kk$, any its submodule is finitely generated.
The group ${\rm Tor}_2^{B}(\kk , Q)$ is zero because augmentation sequence (\ref{aug1}) provides a flat $B$-resolution of length 2 for $\kk$.
\end{proof}

\begin{Corollary} Let $\kk$ be a noetherian ring. Then Poincare groupoid $\kk \Gamma$ and algebra $B(\Gamma )$ are coherent.
\end{Corollary}
\begin{proof} Since the fundamental group $\pi (\Gamma )$ is free, the group ring $\kk [\pi (\Gamma )]$ is quasi-free relatively over $\kk$ (cf. \cite{CQ1}). Indeed, the free algebra $\kk \langle x_1, \dots , x_n \rangle$ is relatively quasi-free over $\kk$ because it obviously satisfies the lifting criterion of proposition \ref{qfcriterion}. The group ring $\kk [\pi (\Gamma )]$ is a localization of the free algebra, hence satisfies the lifting criterion too. By theorem  \ref{thquasifree} it is coherent.
By proposition \ref{moritakg}, algebra $\kk \Gamma$ is Morita equivalent to the group ring of $\kk [\pi (\Gamma )]$. Therefore, it is coherent too. Coherence of $B(\Gamma )$ follows by theorem \ref{bcoherent}.
\end{proof}


This corollary implies that the categories of finitely presented left modules over $\kk \Gamma$ and over $B(\Gamma )$ are abelian.
We indicate the categories of finitely presented (left) modules by the subscript $fp$.

\begin{Proposition}\label{finrep4psi}
Let $\Delta$ be a well-tempered element in algebra $A$ and $B={\widehat A_{\Delta}}$.
Then functors $\psi_{1*}$, $\psi_2^*$, $\psi_1^!$ and $\psi_1^*$ take finitely presented left modules over corresponding algebras
to finitely presented modules.
\end{Proposition}
\begin{proof} A finitely presented $A$-module $W$ has a presentation by finitely generated free $A$-modules. By applying exact functor $\psi_{1*}$ to this presentation, we obtain a presentation for $\psi_{1*}W$ by projective $B$-modules which are finite sums of copies of $B^+$. Since $B^+$ is a finitely generated $B$-module by lemma \ref{fingen}, $\psi_{1*}W$ is finitely presented.

Functors $\psi_1^*$ and $\psi_2^*$ take finitely generated free $B$-modules to  finitely generated free $A$-modules. As the both functors are right exact, it follows that they take finitely presented $B$-modules to finitely presented $A$-modules.

Functor $\psi_1^!$ is isomorphic to $\psi_2^*$ by proposition \ref{naturaliso}.

\end{proof}
\begin{Lemma}\label{kkb}
Let $\kk$ be a noetherian ring. Then any finitely generated $\kk$-module is finitely presented as a $B$-module.
\end{Lemma}
\begin{proof} Since $\kk$ is noetherian every finitely generated module has a presentation by finitely generated free modules over $\kk$. Finitely generated free modules over $\kk$ have finite projective resolution over $B$, because $\kk$ has a projective $B$-resolution (\ref{aug1}). The fact follows.
\end{proof}

\begin{Proposition}\label{fppsi2} Let $\Delta$ be well-tempered and $\kk$ noetherian. Then functor $\psi_{2*}$ takes finitely presented modules to finitely presented if and only if $\{a\in A|\Delta a=0\}$ and  $A/\Delta A$ are finitely generated $\kk$-modules.
\end{Proposition}
\begin{proof} Assume that $\psi_{2*}$ takes finitely presented modules to finitely presented. Consider exact sequence (\ref{psilow12}) applied to module $W=A$. Since $\psi_{1*}W$ is finitely presented and the category of finitely presented modules is abelian, the kernel and cokernel of $\lambda_A$ must be finitely presented modules. But they are $B^+$-trivial modules with the kernel being isomorphic to $\{a\in A|\Delta a=0\}$ and cokernel to $A/\Delta A$. Hence they must be finitely generated $\kk$-modules.

Conversely, assume that these $\kk$-modules are finitely generated. By lemma \ref{kkb} it is enough to show that the kernel and cokernel of $\lambda_W$ are finitely generated $\kk$-modules for a finitely presented $W$. For $W=A$, this is clear from (\ref{psilow12}).


Consider the 4-term exact sequence that comes from a finite presentation of $W$:
$$
0\to Z\to A^k\to A^l\to W\to 0
$$
First, since functor $\psi_{2*}(-)/B^+\psi_{2*}(-)$ is right exact and its value on $A^l$ is a finitely generated $\kk$-module then so is its value on $W$. Thus the cokernel of $\lambda_W$ is a finitely generated $\kk$-module.
Second, by applying exact functor $\psi_{1*}$ to the exact sequence, we get a 4-terms exact sequence:
$$
0\to \psi_{1*}(Z)\to (B^+)^k\to (B^+)^l\to \psi_{1*}(W)\to 0
$$
By applying functor ${\rm Hom}_B(\kk , -)$ to it, we get a spectral sequence that converges to 0. It shows that $(\psi_{1*}W)^{B^+}={\rm Hom}_B(\kk ,\psi_{1*}W)$ is filtered by the $\kk$-subquotients of modules 
$$
{\rm Hom}_B(\kk , (B^+)^k), {\rm Ext}^1_B(\kk , (B^+)^k), {\rm Ext}^2_B(\kk , \psi_{1*}Z).
$$
By lemma \ref{homkb+} the first two are finitely generated $\kk$-modules, while the third module is zero.
Since $\kk$ is noetherian, the kernel of $\lambda_W$ is finitely generated over $\kk$.
\end{proof}

For an $A$-module $W$, consider the map $\lambda_W:\psi_{1*}W\to \psi_{2*}W$ defined by natural transformation (\ref{lowpsi}). Define the $B$-module $W_{min}$ as the image of $\lambda_W$.

\begin{Lemma}\label{trivialcone}
Let $W$ be an $A$-module and $V$ a $B^+$-trivial $B$-module. Then
$B$-module $W_{min}$  satisfies
${\rm Hom}_{B}(V, W_{min})=0$ and ${\rm Hom}_{B}(W_{min}, V)=0$.
\end{Lemma}
\begin{proof}

Module $W_{min}$ is a submodule in $\psi_{2*}W$, hence ${\rm Hom}_{B(\Gamma )}(V, W_{min})=0$ by lemma \ref{zerohoms}(iv). Also it is a quotient of $\psi_{1*}W$, hence ${\rm Hom}_{B(\Gamma )}(W_{min}, V)=0$ by lemma \ref{zerohoms}(iii).
\end{proof}

If $V$ is a $B$-module, then define the {\em minimal shadow} of $V$ by $V_{min}:=(\psi_1^!V)_{min}=(\psi_2^*V)_{min}$.
\begin{Proposition}
Given $B$-module $V$ we have a commutative diagram:
\[
\begin{CD}
\psi_{1*}\psi_1^!V @>\alpha>> V_{min} \\
@V {\epsilon_V} VV @VV \beta V\\
V @>> {\delta_V} > \psi_{2*}\psi_2^*V
\end{CD}
\]
with an epimorphism $\alpha$ and monomorphism $\beta$.
\end{Proposition}
\begin{proof}
This follows from lemma \ref{3composite} and proposition \ref{naturaliso}.
\end{proof}

For
 a
left $B(\Gamma )$-module $V$ and a vertex $i\in V(\Gamma )$, define $\kk$-modules
$$
V_i=\{v\in V|\ \rho (x_i)v=v\}.
$$
If $\Gamma $ is a connected graph, then $V_i$ are isomorphic for all $i\in V(\Gamma )$.
Indeed, if $(ij)\in E(\Gamma )$ and $v\in V_i$, then $s_{ij}^{-1}\rho (x_j)v\in V_j$. Moreover, $s_{ij}^{-1}\rho (x_i)(s_{ij}^{-1}\rho (x_j)v)=s_{ij}^{-2}\rho (x_ix_jx_i)v=v$. Thus, $s_{ij}^{-1}\rho (x_j)$ and $s_{ij}^{-1}\rho (x_i)$ are mutually inverse transformations between $V_i$ and $V_j$. In general, the isomorphism depends on the choice of path connecting vertices $i$ and $j$.

\begin{Proposition}\label{tensorhom}
Let $V$ be a left $B(\Gamma )$-module.
Then there is a natural decomposition:
\begin{equation}
\psi_2^*V=\oplus_i\ V_i
\end{equation}
\end{Proposition}
\begin{proof}
Since $\kk\Gamma$ is isomorphic
to $B^+(\Gamma )$ as a right $B(\Gamma )$-module, we can apply (\ref{bgxiright}) to calculation of
$\kk \Gamma\otimes_{B(\Gamma )} V$.
Since $x_i$ is an idempotent we have a short exact sequence of right modules (actually, a direct sum decomposition):
\begin{equation}\label{bbxi}
0\to (1-x_i)B(\Gamma )\to B(\Gamma )\to x_iB(\Gamma )\to 0
\end{equation}
Taking tensor product of (\ref{bbxi}) with $V$ over $B(\Gamma )$ shows that the
$$
x_iB\otimes_BV=V/{\rm Im}({\rm id}_V-\rho(x_i))=V_i.
$$
\end{proof}
%

We say that a finite dimensional $B(\Gamma )$-module is {\em of rank $r$}, and write ${\rm rank}\ V=r$, if ${\rm dim}V_i=r$.
Rank 1 representations $V$ of $B(\Gamma )$ correspond to configurations of lines in $V$ subordinated to $\Gamma$. The minimal shadow of the representation corresponds to the minimal configuration that was canonically assigned to any configuration in subsection \ref{problem formulation}.
\bigskip

{\bf The end of the proof of theorem 3.} After we have developed the relevant techniques, the assertions of the theorem become natural and easily to check. First, the equivalence of $k[\pi (\Gamma )]$-mod with $k\Gamma$-mod takes 1-dimensional $k[\pi (\Gamma )]$ representations into  $k\Gamma$ representations $W$ in which the idempotents $e_i$ are represented by projectors of rank 1. Further, the functor $W\mapsto W_{\rm min}$ takes such  representations into rank 1 minimal representations of algebra $B(\Gamma )$. These representations correspond exactly to configurations of lines which are considered in the theorem, because the lines are the images of the projectors $x_i$. Since the 1-dimensional representations of a group are its characters and the moduli of the characters of the fundamental group of the graph is ${\rm H}^1(\Gamma. k^*)$, we get that this group  parameterizes the minimal configurations of lines, which are the $S$-equivalence classes of configurations.

\subsection{Duality}

Assume that we have an involutive anti-automorphism $\sigma =\sigma_A$ on $A$ that preserves $\Delta$:
$$
\sigma (\Delta )=\Delta .
$$
Then, $\sigma$ induces an anti-automorphism on $B^+$, hence on $B$, which we denote by $\sigma_B$ when we need to distinguish it from $\sigma_A$. Indeed,
$$
\sigma (a\cdot_{\Delta}b)=\sigma(a\Delta b)=\sigma (b)\sigma (\Delta )\sigma (a)= \sigma(b)\cdot_{\Delta}\sigma (a).
$$
Further, the anti-involutions on $A$ and $B$ satisfy a relation with respect to $\psi_1$ and $\psi_2$:
\begin{equation}\label{psiduality}
\sigma_A\circ \psi_1 =\psi_2\circ \sigma_B .
\end{equation}

Remember that we have an anti-involution (\ref{dualkg}) on the algebra $\kk\Gamma$. It preserves Laplace operator $\Delta$, hence it induces an anti-involution
$\sigma_B : B(\Gamma )\to B(\Gamma )^{opp}$, which is defined by:
$$
\sigma_B (x_i)= x_i.
$$


Any anti-isomorphism of an algebra induces an equivalence between the categories of left and right modules of the algebra. If $\kk = k$ is a field and a left $B$-module is finite dimensional over $k$, then the dual vector space is a right $B$-module. The composition of the equivalence induced by an anti-isomorphism with "taking dual" gives a duality, i.e. an involutive anti-equivalence, on the category of finite dimensional right $B$-modules. Anti-involution $\sigma_B$ induces a duality $D: B(\Gamma )-{\rm mod}\simeq B(\Gamma )-{\rm mod}^{opp}$. For the representation $\rho :B(\Gamma )\to {\rm End}(V)$, the dual representation $D(\rho ) :B(\Gamma )\to {\rm End}(V^*)$ in $V^*$ is defined by
$$
D(\rho )(b)=\rho (\sigma_B (b))^*.
$$
On generators $x_i$'s of $B^+(\Gamma )$ the duality acts by $x_i\mapsto x_i^*$.
Thus, the duality is an algebraic version of the duality discussed in the beginning of section \ref{problem formulation}.

An interesting problem is to study self-dual representations of $B^+(\Gamma )$, i.e representations $V$ that allow an isomorphism with the dual representation: $V\simeq V^*$.

\subsection{Morita equivalence of homotopes}
Consider a pair of elements $\Delta ,\Delta '$ in an algebra $A$. Here we address the problem when the corresponding algebras $B={\widehat A_{\Delta}}$ and $B'={\widehat A_{\Delta '}}$ are Morita equivalent, i.e. when there exists a $B'-B$-bimodule $N$ such that the functor $\Phi_N: {\rm mod}-B\to {\rm mod}-B'$ given by
$$
\Phi_N: V\mapsto N\otimes_{B}V
$$
induces an equivalence on the categories of modules.

We use superscript $'$ for notation of homomorphisms related to $\Delta '$, to distinguish them from similar homomorphisms related to $\Delta$.

Fix elements $c, d\in A$ which satisfy the equation:
\begin{equation}\label{cdeltad}
c\cdot_A\Delta=\Delta '\cdot_A d.
\end{equation}
Let us construct bimodule $N=N_{cd}$ as follows. We want $N$ to be an extension of a trivial bimodule $\kk$ by $A$, the latter is considered as a $B'-B$-bimodule with the left $B'$-module structure coming via $\psi_1'$ and with the  right $B$-module structure coming from $\psi_2$:
\begin{equation}\label{seqbimodule}
0\to A\to N_{cd}\to \kk \cdot z\to 0
\end{equation}
with $z$ being a generator in $\kk=\kk\cdot z$. Take $N=\kk\cdot z\oplus A$ as a vector space. Define the left action of $b'\in B'^{+}$ on $z$ by:
$$
b'z:=b'\cdot_{A} c.
$$
This is an element in $A\subset N_{cd}$. Similarly, define the right action of $b\in B^+$ on $z$ by:
$$
zb:=d\cdot_Ab.
$$
The action of the unit is, of course, by the identity.
\begin{Lemma}
The above defined action endows $N_{cd}$ with a $B'-B$-bimodule structure.
\end{Lemma}
\begin{proof}
Indeed, we have for the left $B'$-module structure:
$$
b_1'(b'z)=b_1'\cdot_{B'}b'\cdot_Ac=(b_1'b')z
$$
and similar for the right $B$-module structure. Also, the left and the right module structures commute:
$$
(b'z)\cdot_Bb=(b'\cdot_Ac)\cdot_Bb=b'\cdot_Ac\cdot_A{\Delta }\cdot_Ab=b'\cdot_A\Delta '\cdot_A d\cdot_Ab=b'\cdot_{B'}(d\cdot_Ab)=b'\cdot_{B'}(zb).
$$
\end{proof}

Let $\Delta '=\Delta$, then $B'=B$. Let $(c, d)=(1_A,1_A)$ be two copies of the unit in $A$. Then $N_{11}=B$ as a $B$-bimodule. Hence the functor $\Phi_{N_{11}}$ is the identity functor in ${\rm mod}-B$.

The composition of functors of type $\Phi_N$ is compatible with products of bimodules:
$$
\Phi_L\circ \Phi_N=\Phi_{L\otimes_{B'}N}
$$
for a $B'-B$-bimodule $L$ and a $B-B''$-bimodule $N$.
\begin{Proposition} Let $\Delta$ be a well-tempered element in $A$.
Let $N_{cd}$ be a $B'-B$-bimodule and $N_{uv}$ be a $B-B''$-bimodule. Then
$$
N_{cd}\otimes_B N_{uv}=N_{c\cdot_Au,d\cdot_Av}
$$
as a $B'-B''$-bimodule.
\end{Proposition}
\begin{proof}
Let $\Delta ,\Delta ', \Delta ''$ be the defining elements for $B, B'$ and $B''$ respectively. By definition of bimodules $N_{cd}$ and $N_{uv}$, we have:
$$
c\cdot_A\Delta=\Delta '\cdot_A d,\ \  u\cdot_A\Delta ''=\Delta \cdot_A v,
$$
which implies that
$$
cu\cdot_A\Delta ''=\Delta '\cdot_A dv.
$$
That is, the module $N_{cu, dv}$ is correctly defined.

We know by (\ref{bbbb}) that $A\otimes_BA=A$ and by lemma \ref{zerohoms} (ii) that $A\otimes_B^{\mathbb L} \kk=0$. By taking tensor product over $B$ of sequences (\ref{seqbimodule}) for $N_{cd}$ and for $N_{uv}$, we obtain that $N_{cd}\otimes N_{uv}$ fits in a similar short exact sequence
$$
0\to A\to N_{cd}\otimes_B N_{uv}\to \kk\cdot (z\otimes w) \to 0,
$$
where $z$ and $w$ stand for generators in the $\kk$ components  for $N_{cd}$ and for $N_{uv}$, respectively. Thus, we need to check  the left action of $B'$ and the right action of $B''$ on $z\otimes w$. Let $b'$ be in $B'$. Then
$$
b'(z\otimes w)=b'\cdot_Ac\otimes w.
$$
Now $b'\cdot_Ac$ is an element in $A\simeq B^+$. According to (\ref{bbbb}), we can decompose it into
$$
b'\cdot_Ac=\sum p_i\cdot_Bq_i,
$$
where $p_i$ and $q_i$ are some elements in $B^+$. Further,
$$
b'\cdot_Ac\otimes w=\sum p_i\cdot_Bq_i\otimes w=\sum p_i\otimes q_iw=\sum p_i\otimes q_i\cdot_Au.
$$
The identification $B^+\otimes_BB^+=B^+$ in (\ref{bbbb}) is given by the multiplication in $B^+$. Hence under this identification, we have
$$
b'(z\otimes w)=\sum p_i\otimes q_i\cdot_Au=\sum p_i\cdot_B q_i\cdot_Au=b'\cdot_Ac\cdot_Au.
$$
Thus the left action of $B'$ on $z\otimes w$ is `via' $c\cdot_Au$. Similarly for the right action of $B''$.
\end{proof}
{\bf Definition.} Two pairs of elements $(c,d)$ and $(c', d')$ in $A$ satisfying (\ref{cdeltad}) are said to be {\em homotopic}, if there exists an element $h\in A$ such that:
$$
c'-c=\Delta 'h, \ \ d'-d=h\Delta .
$$
This clearly defines an equivalence relation.

\begin{Lemma}
If pairs $(c,d)$ and $(c' , d')$ satisfying (\ref{cdeltad}) are homotopic, then $B'-B$-bimodules $N_{cd}$ and $N_{c'd'}$ are isomorphic.
\end{Lemma}
\begin{proof}
The bimodules are extensions of $\kk$ by $A$, hence they are classified by ${\rm Ext}^1_{B'-B}(\kk , A)$. One can use projective bimodule resolution (similar to the one we used to find Hochshild dimension)
$$
0\to (B')^+\otimes B^+\to (B'\otimes B^+)\oplus ((B')^+\otimes B)\to B'\otimes B\to \kk \to 0.
$$
to calculate this group and see that pairs satisfying (\ref{cdeltad}) modulo homotopy equivalence give exactly elements of ${\rm Ext}^1_{B'-B}(\kk , A)$. One can also make change $z\mapsto z+h$ which gives a new presentation for $N_{cd}$, which  identifies it with $N_{c'd'}$.
\end{proof}
\begin{Theorem} If $\Delta$ and $\Delta '$ are well-tempered elements in $A$ and there exist elements $c,d, u, v, h_1,h_2\in A$ satisfying
$$
cu=1+\Delta 'h_1,\ dv=1+h_1\Delta ',\ uc=1+\Delta h_2,\ vd= 1+h_2\Delta ,
$$
then categories $B-{\rm mod}$ and $B'-{\rm mod}$ are equivalent.
\end{Theorem}
\begin{proof}
By proposition 58 and lemma 59, functors $\Phi_{N_{cd}}$ and $\Phi_{N_{uv}}$ are mutually inverse. Hence the equivalence.
\end{proof}
This theorem claims that under condition of well-temperedness the category of representations of the homotope is an invariant of the homotopy class of two term complexes of right $A$-modules
of the form $\{A\to A\}$.

Take tensor product over $B$ of the sequence (\ref{seqbimodule}) with the $B^+$-trivial left $B$-module $\kk $. We know that  $A\otimes_B^{\mathbb L} \kk =0$, therefore $\Phi_{N_{cd}}(\kk)=\kk $.
\begin{Proposition}
If $\Delta$ is well-tempered, then functors $\psi_{1*}$ and $\psi '_{1*}$ are compatible with equivalences $\Phi_{N_{cd}}$, i.e.
$$
\Phi_{N_{cd}}\circ \psi_{1*}=\psi '_{1*}.
$$
\end{Proposition}
\begin{proof}
Let $M$ be a left $A$-module. Take tensor product over $B$ of $\psi_{1*}M$ with short exact sequence (\ref{seqbimodule}). First, 
$$
\kk \otimes_B^{\mathbb L}\psi_{1*}M= \kk \otimes_B^{\mathbb L}A\otimes_AM=0.
$$
Second, 
$$
A\otimes_B\psi_{1*}M=A\otimes_BA\otimes_AM=\psi_{1*}'M,
$$
because $A\otimes_BA=A$ for well-tempered $\Delta$. This implies that $N_{cd}\otimes_B^{\mathbb L}\psi_{1*}M=\psi_{1*}'M$.
\end{proof}


\subsection{The example of the matrix algebra over a field}
\label{examplematrix}
Let $A$ be a matrix algebra over a field $\kk =k$. We identify $A$ with the algebra of operators in a vector space $V$ over $k$ of dimension $n$.
\begin{Proposition}
All non-zero elements in $A$ are well-tempered. If $A$ is an operator of corank $s$, then the homotope $B$ is Morita equivalent to the algebra of the quiver with two vertices and $s$ arrows $\alpha_i$ in one direction  and $s$ arrows $\beta_j$ in the opposite direction and relations $\beta_j\alpha_i=0$, for all $1\le i,j\le s$.
\end{Proposition}
\begin{proof} First, $A$ is simple, hence, for any non-zero element $\Delta$, the two sided ideal $A\Delta A$ coincides with $A$.  Let $\Delta$ be an operator of corank $s$. We know that $B$ remains in the same isomorphism class when $\Delta$ is multiplied by invertible elements in $A$ on both sides.
By multiplying $\Delta$ by invertible operators, we can make it to become a projector $\Delta =p$ of corank $s$.

Projector $p$ can be decomposed into a sum 
$$
p=\sum p_i
$$
of orthogonal (i.e. $p_ip_j=0$) minimal projectors in $A$. It is easy to check that the same decomposition for $p$ into orthogonal idempotents holds in $B^+$. Matrix algebra $A$, as a left module over itself, is decomposed into a sum of $n$ copies of the projective module $Ap_1$. Since $B^+=\psi_{1*}A$ it has also decomposition into $n$ copies of $\psi_{1*}Ap_1=Bp_1$.

Let $q=1_B-p$.  As a left $B$-module, $B$ is decomposed
$$
B=Bp\oplus Bq=\oplus_iBp_i\oplus Bq
$$
which implies that $Bp_i$ are projective modules, hence $B^+$ is projective, whence $\Delta$ is well-tempered.
Consider $B$-module $P=Bp_1\oplus Bq$. Clearly, $P$ is a projective generator in $B-{\rm mod}$. One can easily check that its endomorphism algebra coincides with the one described in the proposition.
\end{proof}

The algebra in proposition 62 has global dimension 2.
In the case of $s=1$, it is the well-known algebra whose category of representations is equivalent to category $\cal O$ of representations of Lie algebra $sl(2, k)$ with fixed central character. Also it is equivalent to the category of perverse sheaves on the 2-dimensional sphere stratified by a point and the complement to it.

\subsection{The example of the cyclic graph}
\label{cyclicgraph}



Let ${\rm H}^1(\Gamma )={\mathbb Z}$. Then the group ring $k[\pi_1(\Gamma )]$ is the algebra $k[x, x^{-1}]$ of Laurent polynomials in one variable. In view of Morita equivalence between $k\Gamma$ and $k[\pi_1(\Gamma )]$, the category $k\Gamma -{\rm mod}$ is equivalent to the category of 
$k[x,x^{-1}]$-mod.


Since $k\Gamma$ is isomorphic to the matrix algebra over $k[\pi_1(\Gamma )]$ (see proposition \ref{moritakg}), elements of $k\Gamma$ can be understood as homomorphisms of free sheaves of $k[x,x^{-1}]$-modules 
of rank $n=|V(\Gamma )|$. In particular, Laplace operator gives a homomorphism 
$$
\Delta : k[x,x^{-1}]^{\oplus n}\to k[x,x^{-1}]^{\oplus n}.
$$


If $\Gamma$ is a cyclic graph with $n$ vertices and $n$ edges $l_{i,i+1}$, where $i \in \ZZ / n$, then
$$
\Delta = 1 + \sum_{i \in \ZZ / n} s_{i,i+1} (l_{i,i+1}+l_{i+1,i}),
$$
for some $s_{i,i+1} \in k^*, i \in \ZZ / n$.
The matrix of $\Delta$, wich we also call Lapaplce operator, in a suitable basis has the form:
\begin{equation}
\label{laplcyc}
\begin{pmatrix}
1 & s_{1,2} & 0 & ... & 0 & s_{n,1}x^{-1}\\
s_{1,2} & 1 & s_{2,3} & 0 & ... & 0\\
0 & s_{2,3} & 1 & s_{3,4} &  ... & 0\\
... & ... & ... & ... & ... & ... \\
0 & ... & 0 & s_{n-2, n-1} & 1 & s_{n-1,n} \\
s_{n,1}x & 0 & ... & 0 & s_{n-1,n} & 1
\end{pmatrix}
\end{equation}

The rank of this operator is exactly the dimension of the minimal representation of the homotope corresponding to character $x$ of the fundamental group. The minor of the size $(n-2)$x$(n-2)$ which contains all columns except for the first and the last ones and all rows except for the last two, is nonzero for any value of $s_{ii+1}$. Hence the rank of the Laplace operator is never less than $n-2$. Thus, the stratification of the moduli space of rank 1 representations of the homotope $B(\Gamma )$ has maximally three strata:
$$
{\mathcal M}(\Gamma)={\mathcal M}^n\cup {\mathcal M}^{n-1}\cup {\mathcal M}^{n-2}, 
$$
which correspond to the minimal representations of dimension $n$, $n-1$ and $n-2$. Here ${\mathcal M}(\Gamma )$ is isomorphic to ${\mathbb A}^1\setminus 0$, parameterized by $x\ne 0$.

It is easy to see that for $n\ge 3$ the determinant of the Lapace matrix has the form:
$$
{\rm det} = A(x+x^{-1})+B,
$$
where $A$ and $B$ are polinomials in $s_{i,i+1}$. Moreover, $A=\pm \prod s_{ii+1}$. The zeros of the determinant define the values of the variables $s_{ii+1}$ and $x$ for which the rank of the Laplacian drops at least by 1. Thus for the generic value of $s_{ii+1}$, we have two values of $x$ for which the rank of the Laplace operator drops by 1. 
These are the coordinates of the two points which comprise ${\mathcal M}^{n-1}$. For the generic case, ${\mathcal M}^{n-2}=\emptyset$.

Computations show that, for special values of $s_{ii+1}$, which are determined by three equations in variables $r_{ii+1}=s_{ii+1}^2$ (whose explicit form me omit here) and define the variety ${\mathcal R}$ of dimension $n-3$, the rank of the Laplace operator drops by 2 in the point $x=1$ or $x=-1$ depending of a connected component of ${\mathcal R}$. Thus, for point in  ${\mathcal R}$, the stratum ${\mathcal M}^{n-1}$ is empty, while ${\mathcal M}^{n-2}$ consists of one point.

\section{Derived categories, recollement and generalised homotopes}
In this section, we approach the subject using the point of view of derived categories. This allows us to understand the homotope construction in terms of the abstract recollement of t-structures. Since the recollement originally was a homological tool for constructing perverse sheaves on stratified spaces, it is not a surprise that, under suitable conditions on stratifications, the categories of perverse sheaves might be equivalent to the categories of representations of homotopes. 

\subsection{Decomposition of the derived category of the homotope}
We assume here that $\Delta\in A$ is well-tempered and $B={\widehat A_{\Delta}}$.
Consider derived categories $D(A-{\rm mod})$ and $D(B-{\rm mod})$.
Functors $\psi_{1*}$ and $\psi_{2*}$ and
functors $\psi_2^*\simeq \psi_1^!$ are exact. We denote by the same symbols the corresponding derived functors.
We show that they provide semi-orthogonal decompositions for $D(B-{\rm mod})$.
Natural transformations (\ref{lowpsi}) and (\ref{naturalmu}) and the fact that $\mu$ is an isomorphism of functors carry on to the context of derived categories.




Denote by $D_0(B-{\rm mod})$ the full subcategory in $D(B-{\rm mod})$ of complexes with cohomology $B^+$-trivial modules. In view of proposition \ref{exceptional}, it is equivalent to $D(\kk -{\rm mod})$. Denote by $i_*:D(\kk-{\rm mod})\to D(B -{\rm mod})$ the corresponding embedding functor.

Recall some definitions from \cite{Bon1}.
A triangulated subcategory is said to be {\em right (resp., left) admissible} if it has right (resp. left) adjoint to the embedding functor.


\begin{Proposition} Subcategory $D_0(B-{\rm mod})$ in $D(B-{\rm mod})$ is left and right admissible.
Functors $\psi_{1*}$ and $\psi_{2*}$ are fully faithful and identify category $D(A -{\rm mod})$ with,
respectively, left and right orthogonal to the subcategory $D_0(B-{\rm mod})$ in $D(B-{\rm mod})$.
\end{Proposition}
\begin{proof}

Exact functors $\psi_{1*}$ and $\psi_{2*}$ between triangulated categories 
are fully faithful if the adjunction morphisms are isomorphisms:
${\rm id}\simeq \psi_2^*\psi_{2*}$, ${\rm id}\simeq \psi_1^!\psi_{1*}$. This follows from proposition \ref{adjunction}, because functors $\psi_2^*=\psi_1^!$, $\psi_{2*}$ and $\psi_{1*}$ are exact as functors between abelian categories.

Let $W$ be an $A$-module and $V_0$ a $B^+$-trivial $B$-module. We already know in view of  lemma \ref{zerohoms} that 
${\rm Hom}_{B(\Gamma)}^{\bullet}(\psi_{1*}W, V_0)=0$ and
${\rm Hom}_{B(\Gamma )}^{\bullet}(V_0, \psi_{2*}W)=0$.
To prove that $D_0(B-{\rm mod})$ is left admissible and, simultaneously, that the image of functor $\psi_{1*}$ is indeed the whole left orthogonal to $D_0(B-{\rm mod})$, it is enough to check \cite{Bon1} that every object, say $V$, in $D(B-{\rm mod})$ has a decomposition into exact triangle $U\to V\to W$ with $U$ in the image of $\psi_{1*}$ and $W\in D_0(B-{\rm mod})$.
Consider the adjunction morphism $\psi_{1*}\psi_1^!V\to V$.
By lemma \ref{epsilondelta} its cone is in $D_0(B-{\rm mod})$ if $V$ is a pure $B$-module.
This gives a triangle with required properties for such $V$.
For general $V$, this follows from exactness of functors $\psi_{1*}$ and $\psi_1^!$. For $\psi_{2*}$, the proof is similar.
%
%
\end{proof}
In accordance with ideology of \cite{Bon1}, this proposition is interpreted as the existence of decompositions into the semiorthogonal pairs:
\begin{equation}\label{semiordec}
D(B-{\rm mod})=\langle i_*D(\kk-{\rm mod}), \psi_{1*}D(A-{\rm mod})\rangle =\langle \psi_{2*}D(A-{\rm mod}), i_*D(\kk-{\rm mod})\rangle.
\end{equation}

Given an admissible subcategory $\cal B$ in a triangulated category $\cal D$, there is an equivalence between the left and right orthogonals to it, $\ ^{\perp} {\cal B}$ and ${\cal B}^{\perp}$. The mutually inverse 'mutation' functors $L_{\cal B}:\ ^{\perp} {\cal B}\to {\cal B}^{\perp}$ and $R_{\cal B}:{\cal B}^{\perp}\to \ ^{\perp}{\cal B}$ are given by restricting to $\ ^{\perp} {\cal B}$ the left adjoint to the embedding functor ${\cal B}^{\perp}\to {\cal B}$ and by restricting to ${\cal B}^{\perp}$ the right adjoint functor to the embedding $\ ^{\perp} {\cal B}\to {\cal D}$.

In the case of our interest,
subcategories $D_0(B-{\rm mod})^{\perp}$ and $\ ^{\perp}D^b_0(B-{\rm mod})$ are both equivalent to $D(A -{\rm mod})$ via functors $\psi_{1*}$ and $\psi_{2*}$, with equality
$$
\psi_{2*}=i_{ {\cal B}^{\perp }}\circ L_{\cal B}\circ \psi_{1*},
$$
where $i_{{\cal B}^{\perp }}$ is the embedding functor for the subcategory ${\cal B}^{\perp }$. Since $L_{\cal B}$ is basically the adjoint to $i_{ {\cal B}^{\perp }}$ we have the adjunction morphism:
\begin{equation}\label{lambdamut}
\lambda :{\rm id}_{^{\perp}{\cal B}}\to i_{{\cal B}^{\perp }}\circ L_{\cal B}.
\end{equation}
It implies the functorial morphism $\psi_{1*}\to \psi_{2*}$, which coincides with the natural transformation $\lambda$ in  (\ref{lowpsi}) when
extended to a transformation between the derived functors.

The embedding functor $i_*$
has the left adjoint 
$$
i^* :D(B-{\rm mod})\to D(\kk -{\rm mod})
$$
defined by
$$
i^*(V)=\RR {\rm Hom}_{B}(V, \kk)^*
$$
and the right adjoint $i^!$ defined by
$$
i^!(V)=\RR {\rm Hom}_{B}(\kk, V).
$$
All together, functors $\psi_{1*}, \psi_{2*}, \psi^!_1, i_*, i^*, i^!$ fit the formalism of six functors of the recollement.

Given an admissible subcategory ${\cal B}$ in a triangulated category,
one can consider the ambient category as being 'glued' from the subcategories ${\cal B}$ and its orthogonal $\ ^{\perp}{\cal B}$. The necessary extra data to define the gluing is the functor
\begin{equation}\label{functorglue}
F: \ ^{\perp}{\cal B}\to {\cal B},
\end{equation}
which takes an object from $\ ^{\perp}{\cal B}$ to the cone of the natural transformation (\ref{lambdamut}) applied to this object.
The resulting object lies in ${\cal B}$.

To be more precise, this approach works only in the framework of pre-triangulated DG-categories \cite{BK2}, rather than for ordinary triangulated categories. Thus we should consider suitable DG-categories, {\em enhancements} for ${\cal B}$ and for its orthogonal $\ ^{\perp}{\cal B}$, and a gluing DG-functor (or DG-bimodule) between them. Given this data, one can construct in an essentially unique way a new DG-category, such that its homotopy category has a semiorthogonal decomposition into a pair $\langle {\cal B}, \ ^{\perp}{\cal B}\rangle$ ({\emph cf.} \cite{KL}, \cite{Ef}, \cite{Or}). 

When applied to our categories, this means that category $D(B-{\rm mod})$ is glued from dg-enhanced categories $D(A -{\rm mod})$ and $D(\kk -{\rm mod})$ by means of a functor which can be understood as a DG $\kk - A$-bimodule, i.e. simply as a right $A$-module, because the left $\kk$-module structure will be automatically induced. This right DG $A$-module is given by the cone of the natural transformation $\lambda$ in (\ref{lowpsi}).

The first semiorthogonal decomposition in (\ref{semiordec}) restricts to the category of bounded complexes with finitely presented cohomology $D^b_{fp}(B-{\rm mod})$, but the second one restricts only when $\psi_{2*}$ takes finitely presented modules to finitely presented. The conditions for this is given in proposition \ref{fppsi2}.

{\bf Example.} Assume that graph $\Gamma$ is a tree and $\kk = k$ is a field. Then homology of $\Gamma$ is trivial and $k\Gamma$ is isomorphic to the algebra of matrices of size $n\times n$ with $n$ equal to the number of vertices in the graph. This example we already considered in \ref{examplematrix}.
The category $k\Gamma -{\rm mod}$ is equivalent to the category of $k$-modules.
The gluing functor $F: D(k\Gamma -{\rm mod})\to D(k-{\rm mod})$ in (\ref{functorglue}) for this case is fully defined by its values on the standard representation of the matrix algebra $\kk\Gamma -{\rm mod}$ in the $n$ dimensional vector space $V$. The value of $F$ on this representation is the cone of  $\lambda_V$, which is nothing but the Laplace operator $\Delta$ understood as an element of the matrix algebra.

Hence,
if $\Delta$ has corank $s$, then $D^b(B(\Gamma -{\rm mod_{fd}}))$ has a semiorthogonal pair, which is basically a full  exceptional collection with two elements $(E_1, E_2)=(i_*\kk , \psi_{1*}P)$, and Ext-groups are given by formulas: 
$$
{\rm Ext}^{0,1}(E_1,E_2)=k^s, \ {\rm Ext}^{\ne 0,1}(E_1,E_2)=0.
$$


\subsection{Recollement}
\label{recolle}
Given an admissible subcategory in a triangulated category and two t-structures, one in a subcategory and another one in its orthogonal, on can define a t-structure in the ambient category by the procedure known as {\em recollement} \cite{BBD}. In this subsection, we show that the standard t-structure in $D(B-{\rm mod})$ is obtained by recollement from the standard t-structure on $D(A -{\rm mod})$ and $D(\kk -{\rm mod})$.

The initial data for recollement are 3 triangulated categories $D$, $D_U$ and $D_F$ and six exact functors between them.
Two of the functors are:
$$
i_*:D_F \to D,\ \ j^*:D\to D_U.
$$
The other functors $i^*, i^!, j_*, j_!$ are the left and right adjoints to these two. They must satisfy a number of constraints which are equivalent to saying that $D_F$ is an admissible subcategory in $D$ and $D_U$ is the quotient of $D$ by $D_F$, with $i_*$ the embedding functor and $j^*$ the quotient functor.

Given a t-structure $(D_U^{\le 0}, D_U^{\ge 0})$ in $D_U$ and a t-structure $(D_F^{\le 0}, D_F^{\ge 0})$ in $D_F$, the glued t-structure is defined by:
$$
D^{\le 0}:=\{K\in D|\ j^*K\in D_U^{\le 0}\ {\rm and}\ i^*K\in D_F^{\le 0}\},
$$
$$
D^{\ge 0}:=\{K\in D|\ j^*K\in D_U^{\ge 0}\ {\rm and}\ i^!K\in D_F^{\ge 0}\}.
$$
According to theorem 1.4.10 in \cite{BBD}, this pair of categories indeed defines a t-structure.

Recall some facts from the general theory of gluing of $t$-structures.
Denote by $\mathcal{A}_F$ and $\mathcal{A}_U$ the hearts of $t$-structures on $D_F$ and $D_U$ respectively and by $\mathcal{A}$ the heart of the glued $t$-structure. These are abelian categories, which are related by six functors, the restriction of the functor $i_*$ to $\mathcal{A}_F$, the restriction of $j^*$ to $\mathcal{A}$ and their right and left adjoint functors to the corresponding hearts. We use for them the same symbols as for the notation of the functors between triangulated categories. Two of these functors, namely
$$
i_*:\mathcal{A}_F\to \mathcal{A},\ \  j^*:\mathcal{A}\to \mathcal{A}_U,
$$
are exact functors between abelian categories. The other ones are exact, in general, only from one side, which is determined by their adjuntion with these two functors. The functors satisfy the properties which resemble those which hold for the functors of triangulated recollement. Altogether  three categories and six functors between them satisfying the above mentioned properties comprise what is called 'abelain recollement' \cite{Kuh}, \cite{FP}. 

In the case of the abelian recollement there is one more important functor,  
{\em the functor of intermediate extension} $j_{!*}:\mathcal{A}_U\to \mathcal{A}$ \cite{BBD}. Functors  $j_!$ and $j_*$ are connected by natural transformation $j_!\to j_*$, which follows from the adjunction of these functors with $j^*$. The intermediate extension is defined as the image of this natural transformation evaluated on an arbitrary object   $A$ in $\mathcal{A}_U$:
$$
j_{!*}(A):={\rm Im}(j_!A\to j_*A).
$$
This functor allows us, in particular, to describe indecomposable objects in the category 
$\mathcal{A}$. Namely, all indecomposable objects in $\mathcal{A}$ are either objects of the form $i_*(S_F)$, where $S_F$ is any indecomposable object in $\mathcal{A}_F$, or objects of the form $j_{!*}(S_U)$, where $S_U$ is any indecomposable object in $\mathcal{A}_U$.

Now let us see how this theory is applicable to homotopes. As we got in section 8.1, category $D(B-{\rm mod})$ for the homotope $B$ constructed by a well-tempered element $\Delta$  has an admissible subcategory $D_F=D(\kk -{\rm mod})$, with the quotient category (the orthogonal) being the category equivalent to $D_U=D(A -{\rm mod})$.

We denote by $(D_F^{\le 0}, D_F^{\ge 0})$ and by $(D_U^{\le 0}, D_U^{\ge 0})$ the standard t-structure in these categories.

The dictionary between our functors and the standard notations for six functors is as follows:
$$
\psi_{1*}\longleftrightarrow j_!,\ 
\psi_{2*}\longleftrightarrow j_*,\ 
\psi_1^!=\psi_2^*\longleftrightarrow j^*,
$$
while the notation for $i_*, i^!, i^*$ coincide.

\begin{Theorem}
The standard t-structure in $D(B-{\rm mod})$ coincides with the one obtained by recollement of the standard  t-structures in $D(\kk -{\rm mod})$ and in $D(A -{\rm mod})$.
\end{Theorem}
\begin{proof} Denote by $(D^{\le 0}, D^{\ge 0})$ the standard t-structure in $D(B-{\rm mod})$ and by $D_{gl}^{\le 0}\cap D_{gl}^{\ge 0}$ the glued t-structure.
The embeddings $D^{\le 0}\subset D_{gl}^{\le 0}$ and $D^{\ge 0}\subset D_{gl}^{\ge 0}$ follow from the t-exactness (in the standard t-structures) of $\psi_1^!$ and from the right t-exactness of $i^*$ and left t-exactness of $i^!$. The inverse inclusions follow from the fact that both $(D^{\le 0}, D^{\ge 0})$ and $D_{gl}^{\le 0}\cap D_{gl}^{\ge 0}$ are t-structures, in particular, $D^{\le 0}$ is the left orthogonal to $D^{\ge 0}$.
\end{proof}

Let us stress the special property of the homotope in the context of the gluing of t-structures. In general, the functors $j_!$ and $j_*$ are respectively left and right exact only, while in the case of the homotope they are equal to $\psi_{1*}$ and $\psi_{2*}$ respectively, hence exact. 

Let $V$ be a $B$-module. We say that it is {\em minimal} if
\begin{itemize}

\item[(i)] ${\rm Hom}_{B(\Gamma )}(V, k)=0,$
\item[(ii)] ${\rm Hom}_{B(\Gamma )}(k, V)=0.$
\end{itemize}
In the case of $B=B(\Gamma )$, this definition is in compliance with the definition of minimal configurations of projectors given in subsection \ref{problem formulation}.

An object $V\in D^b(B-{\rm mod})$ is said to be an {\em extension} of an object $W\in D^b(A -{\rm mod})$ if $\psi_1^!V\simeq W$. Given a $A$-module $W$, there is a unique up to isomorphism {\em minimal extension} of $W$. It can be defined as the image:
$$
W_{min}:={\rm Im}\ \lambda_W :\psi_{1*}W\to \psi_{2*}W.
$$
In other words, the minimal extension is nothing but the intermediate extension of an $A$-module in the glued t-structure. 

Given an object $V\in D^b(B-{\rm mod})$, we define its {\em minimal shadow} as the minimal extension for the $A$-module $\psi^!_1V$:
$$
V_{min}=(\psi^!_1V )_{min}={\rm Im}\ \lambda_{\psi^!_1V} :\psi_{1*}\psi^!_1V\to \psi_{2*}\psi^!_1V.
$$
Clearly, the assignment $V\mapsto V_{min}$ is extended to a functor. Note that $\lambda_{\psi^!_1V}$ is identified with the composite of adjunction maps:
$$
\psi_{1*}\psi^!_1V\to V \to \psi_{2*}\psi^*_2V
$$
In the language of the gluing of t-structures, the functor of the minimal shadow is nothing but $j_{!*}j^*$. Thus, the consruction of the minimal configuration of lines in a vector space in section 2.3 gets the functorial meaning in the terms of the gluing of t-structures.

\subsection{Homotopes out from the dg gluing}
Since the homotope provides the t-structure obtained by gluing from a recollement on the components of a semiorthogonal pair, it makes sense to reverse the logic and see how the homotope naturally appears in dg gluing.

Given an associate algebra $C$ over a commutative ring $\kk$, free as a $\kk$-module, we glue (semiorthogonally) its derived category of left modules with the category of $\kk$-modules.  In the case, when $\kk$ is a field, this can be considered as an extension of the triangulated category by one exceptional object.

Since the operation of gluing of categories is well defined in the dg world only, we assume to work with dg enhancements of the categories of modules. The glued category is uniquely defined by a complex of {\bf right} $C$-modules, N, because its left $\kk$-module structure is implied by the right $C$ action.

Consider $N$, a complex of right $C$-modules, which has non-trivial components in degree 0 and 1 only. We assume that module $N_1$ is a free right $C$-module of finite rank, say  $N_1=V_1\otimes_{\kk } C$, where $V_1$ is a free $\kk$-module. Note that if $N$ was a complex of right $C$-modules with cohomology in degree 0 and 1 only, such that the module of the 1-st cohomology is finitely generated over $C$, then we can always find a another complex with the above discussed properties, quasiisomorphic to $N$.
Denote by $\Delta$ the operator
$N_0\to V_1\otimes_{\kk} C$ that defines the differential in this complex.

The glued DG category is the category of DG modules over algebra $D$ which can be understood as the algebra of $2\times 2$ matrices whose entries are $d_{11}\in \kk$, $d_{12}\in N$, $d_{21}=0$, $d_{22}\in C$. Thus, algebra $D$ can be presented in the matrix form: 
\begin{equation}\label{matrixd}
D=
\left(\begin{array}{cc}
\kk&N\\
0&C
\end{array}\right)
\end{equation}
The columns of this matrix have the meaning of left DG $D$-modules. We denote them by $P_1$ and $P_2$.

It is known  that
$$
{\rm Hom}(P_i, P_j)=D_{ij},
$$
where $D_{ij}$ are entries of the matrix (\ref{matrixd}) for $D$. Taking this into account, one can calculate the degree zero component of the homomorphism group
$$
{\rm Hom}^0(P_1, P_2\otimes_{\kk}V_1^*[1])={\rm End}_{\kk}V_1\otimes_{\kk}C,
$$
where $V_1^*={\rm Hom}_{\kk}(V_1, \kk )$.
Define the element $\eta = {\rm id}_{V_1}\otimes 1_C$ in this group. Note that it depends on presentation of the degree 1 component of $N$ in the form $V_1\otimes_{\kk} C$.
Consider the twisted complex $P$ over $D$ \cite{BK2} of the form
\begin{equation}\label{twistcomplex}
P_1\to P_2 \otimes_{\kk}V_1^*[1]
\end{equation}
with the defining morphism  $\eta$.
Denote by ${\tilde B}$ the dg algebra of endomorphisms of the twisted complex $P$. An easy calculation shows that ${\tilde B}$ can be presented in the matrix form as:
\begin{equation}\label{matrixb}
{\tilde B}=
\left(\begin{array}{cc}
\kk&{\rm Hom}_{\kk}(V_1, N)\\
0&C\otimes_{\kk}{\rm End}_{\kk}(V_1)
\end{array}\right)
\end{equation}
Note that all the entries of this matrix, except for ${\tilde B}_{12}$, live in the degree 0 component, while ${\tilde B}_{12}$ has the degree 0 component
${\rm Hom}_{\kk}(V_1, N_0)$ and the degree 1 component ${\rm Hom}_{\kk}(V_1, V_1)\otimes_{\kk}C$.

If
$
\left(\begin{array}{cc}
\lambda&a\\
0&e
\end{array}\right)
$ is an element of degree 0 in this algebra, i.e. $a\in {\rm Hom}_{\kk}(V_1, N_0)$, then the differential is defined by:
$$
\partial
\left(\begin{array}{cc}
\lambda&a\\
0&e
\end{array}\right)
=
\left(\begin{array}{cc}
0&\lambda\cdot 1-e+\Delta a\\
0&0
\end{array}\right)
$$
This shows that the differential is epimorphic (essentially by means of $e$) on the degree 1 component of the algebra. Hence the DG algebra ${\tilde B}$ is quasiisomorphic to its degree 0 homology, which is
an ordinary algebra. By eliminating $e$:
$$
e=\lambda \cdot 1+\Delta a,
$$
we get that elements in ${\rm H}^0B$ are described by pairs
\begin{equation}\label{genhomotope1}
(\lambda , a)\in \kk\times {\rm Hom}_{\kk}(V_1, N_0)
\end{equation}
By multiplying matrices of the form
$$
\left(\begin{array}{cc}
\lambda&a\\
0&\lambda\cdot 1+\Delta a
\end{array}\right),
$$
we see that the multiplication in $B={\rm H}^0({\tilde B})$ is given by
\begin{equation}\label{genhomotope2}
(\lambda_1 , a_1)\cdot (\lambda_2 , a_2)= (\lambda_1\lambda_2, \lambda_1 a_2+\lambda_2a_1+a_1\Delta a_2)
\end{equation}
We see that this formula defines the homotope when $N_0=N_1$.

We have proven the following
\begin{Proposition}\label{dghomotope}
Under the assumption that the gluing dg $C$-module $N$ has two components, in degree 0 and 1, and the degree 1 component is a free right $C$-modules of finite rank,
the endomorphism algebra of $P$ is quasiisomorphic to its 0-th homology, and ${\rm H}^0({\rm End}P)$ is the algebra defined by (\ref{genhomotope1}) and (\ref{genhomotope2}).
\end{Proposition}

\subsection{Generalized homotopes}

The DG gluing from the previous paragraph suggests a generalization of homotopes to the case when $\Delta$ is not an element of any algebra, but a homomorphism of right $C$-modules:
\begin{equation}\label{deltann}
\Delta :N_0\to N_1.
\end{equation}
Define $B^+$ to be the space of morphisms $N_1\to N_0$ and endow it with the multiplication $a\cdot b=a\Delta b$. We adjoin the unit to this space, as we did when $A$ was an algebra, and get an algebra $B$ with the maximal ideal $B^+$. 

We don't have algebra homomorphisms $\psi_1, \psi_2: B\to A$ any more, instead we have algebra homomorphisms
$$
\psi_0: B \to A_0={\rm End}_C(N_0),\ \ \psi_1: B \to A_1={\rm End}_C(N_1),
$$
defined on $B^+$ by
$$
\psi_1(b)=b\Delta ,\ \ \psi_2(b)=\Delta b,
$$
and the unit in $B$ goes into the unit in $A_i$.

By construction, the ideal $B^+$ has a natural structure of $A_0-A_1$-bimodule given by usual composition of operators.

We have functors $\psi_{0*}:A_0-{\rm mod}\to B-{\rm mod}$, $\psi_{1*}:A_1-{\rm mod}\to B-{\rm mod}$ and there adjoints.

We can extend the notion of {\em well-temperedness} to the case when  $\Delta$ is of the above form
by taking literally the same definition as in the case when $\Delta$ was an element of an algebra.

The following proposition generalizes the example from \ref{examplematrix}.
\begin{Proposition}
Let $\kk =k$ be a field, $C=k$ and $\Delta :V_0\to V_1$ a non-zero operator between finite dimensional vector spaces with kernel of rank $s$ and cokernel of rank $t$. Then $\Delta$ is well-tempered and $B$ is Morita equivalent to the algebra of the quiver with two vertices and $s$ arrows $\alpha_i$ in one direction  and $t$ arrows $\beta_j$ in the opposite direction and relations $\beta_j\alpha_i=0$, for all pairs $(i,j)$.
\end{Proposition}
\begin{proof} Denote by $A$ the space of operators $V_1\to V_0$. It is easy to show that $A\Delta A=A$, which is the first condition for $\Delta$ to be well-tempered. Let $r$ be the rank of $\Delta$. One can easily find $r$ rank 1 operators in $A$ such that:
$$
p_i\Delta p_i=p_i, \ \ {\rm and}\ \ p_i\Delta p_j=0\ \ {\rm for}\ \  i\ne j.
$$
Denote be $x_i$ the corresponding elements in $B^+$. They are pairwise orthogonal idempotents in $B$. Modules $Bx_i$ are projective and easily seen to be pairwise isomorphic. Let 
$$
y=1_B-\sum x_i,
$$
then $y$ is an idempotent orthogonal to all $x_i$'s.
We have a decomposition $B=\oplus_iBx_i\oplus By$.
Since $A$, as a left $A_0$-module, is isomorphic to a direct sum of ${\rm dim}V_1$ copies of the module $A_0p_1$, then $B^+=\psi_{1*}A$ is isomorphic to the sum of ${\rm dim}V_1$ copies of $Bx_1$. Therefore, $B^+$ is projective, i.e. $\Delta$ is well-tempered. Module $P=Bx_1\oplus By$ is a projective generator in B-mod. The endomorphism algebra of $P$ is isomorphic to the quiver described in the proposition.
\end{proof}
The homotope from proposition  66 has global dimension 2, hence by Dlab-Ringel theorem it is a  quasi-hereditary algebras \cite{DR}. These generalised homotopes were considered by W.P.Brown under the name of generalized matrix algebras in \cite{Br}.

\subsection{Recollement for generalized homotopes}

Assume again that we are given a $\kk$-algebra $C$ and a complex of right $C$-modules $N$ with cohomology in degree 0 and 1 only. By choosing a representative in the class of this object up to derived equivalence, we can assume that $N$ has only two terms in degree 0 and 1 and the degree 1 term $N_1$ is a free $C$-module: 
$$
N_1=V_1\otimes_{\kk}C.
$$

We consider the DG algebra $D$ of the form (\ref{matrixd}). The derived category of left DG  $D$-modules has a semiorthogonal decomposition into the category of $\kk$-modules and $C$-modules. Denote by $i_*: D(\kk -{\rm mod})\to D(D-{\rm mod})$ and $j^*: D(D -{\rm mod})\to D(C-{\rm mod})$ the natural embedding and projection functors. One can easily check that these functors define a recollement data for gluing of t-structures \cite{BBD}, meaning that the left and right adjoint functors to $i_*$ and $j^*$ are well-defined. Since both categories $D(\kk -{\rm mod})$ and $D(C -{\rm mod})$ have natural t-structures, we can glue them into a t-structure on $D(D -{\rm mod})$.

We will also consider the triangulated category $D-{\rm Perf}$, of perfect DG $D$-modules. By definition (see \cite{BLL}), this is the thick (i.e. closed with respect to taking direct summands) triangulated envelop of the images of twisted complexes in the derived category of $D$-modules. This definition is applicable to usual algebras as well.

Assume that $V_1$ is a free module of finite rank over $\kk$.
We consider the twisted complex $P$ defined by (\ref{twistcomplex}). According to the theory of twisted complexes \cite{BK2}, we have an exact triangle in $D-{\rm mod}$:
\begin{equation}\label{triangleforp}
P_2 \otimes_{\kk}V_1^*\to P\to P_1\to P_2 \otimes_{\kk}V_1^*[1]
\end{equation}
Since both ends of the triangle are in $D-{\rm Perf}$, then $P$ descends to an object in $D-{\rm Perf}$. According to proposition \ref{dghomotope}, $P$ is a tilting module and the endomorphism algebra of $P$ in $D-{\rm Perf}$ is the generalized homotope $B$ constructed by means of the differential $\Delta$ (\ref{deltann}).

\begin{Theorem}
Let $\Delta: N_0\to N_1$ be well-tempered and $N_1=V_1\otimes_{\kk}C$ a free right $C$-module with $V_1$ a free $\kk$-module of finite rank. Then:
\begin{itemize}
\item
The embedding functor provides equivalences of triangulated categories:
${\rm Perf}-B=D-{\rm Perf}$ and $D({\rm mod}-B)\to D(D-{\rm mod})$,
\item
The standard t-structure on $D({\rm mod}-B)$ transported by this equivalence onto $D(D-{\rm mod})$ coincides with the one obtained by recollement from the standard t-structures on $D({\rm mod}-\kk )$ and $D({\rm mod}-C)$.
\item
The object $P$ defined by (83) is a perfect generator in the category $D(D-{\rm mod})$.
\end{itemize}
\end{Theorem}
\begin{proof}
According to \cite{BLL}, to ensure that the functors ${\rm Perf}-B\to D-{\rm Perf}$ and $D({\rm mod}-B)\to D(D-{\rm mod})$ provide equivalences, we need to show that the thick envelop of $P$ is $D-{\rm Perf}$.
We have a semiorthogonal decomposition of $D-{\rm Perf}$ into $\kk -{\rm Perf} $ and $C-{\rm Perf}$, similar to the one in the case of unbounded complexes. The components of the decomposition are generated by $P_1$ and $P_2$, respectively. Consider the functor ${\rm Perf}-B\to D-{\rm Perf}$ that takes the rank 1 free $B$-module to $P$. It takes the sequence (50) into the exact triangle (88). Since $B^+$ is a projective module, it belongs to ${\rm Perf}-B$. Hence its image $P_2\otimes_{\kk}V_1$ belongs to the thick envelop of $P$.
Hence, $P_2$ itself belongs to this envelop. The triangle (88) implies that $P_1$ is in the same category too. Thus, the thick envelop of $P$ coincides with $D-{\rm Perf}$.

It follows from the t-exactness of $i_*$ and $j^*$ that the standard t-structure is glued. 

\end{proof}

Let $C=\kk\Gamma $ be the Poincare groupoid and $\Delta\in \kk\Gamma$ the generalized Laplace operator. Since $\kk\Gamma$ is isomorphic to the matrix algebra over $C$ we can regard $\Delta$ as a morphism of right $C$-modules $C^n\to C^n$. Consider a complex $N$ of right $C$-modules with two non-trivial components $C^n$ in degree 0 and 1 with differential $\Delta$ between them and apply the construction from the previous paragraph. We get the DG algebra $D$ and object $P$ in D-mod.

\begin{Corollary}
We have an equivalence of triangulated categories:
$$
B(\Gamma )-{\rm Perf}= D-{\rm Perf}
$$
\end{Corollary}
\begin{proof}
The element $\Delta$ is well-tempered according to lemma 18. Therefore, we can apply theorem 67.
\end{proof}

\section{Perverse sheaves}
In this section we discuss how representation theory of homotopes provides a description of perverse sheaves on some stratified spaces.

Consider a topological space $X$ together with a stratification by a closed set $F$ and its open complement $U$. Let $D^b(X)$ be the derived category of complexes of sheaves on $X$ with cohomology sheaves constructible with respect to this stratification, i.e. locally constant along the strata. $D^b(X)$ has a semiorthogonal decomposition into a pair of subcategories, $D^b(U)$ and $D^b(F)$, which are the derived categories of complexes of sheaves with locally constant cohomology on $U$ and $F$ respectively. It will be sufficient for us to consider the case when $F$ is contractible, say just a point. If $C$ is the group algebra of the fundamental group of $U$ and $U$ is good enough, then $D^b(U)=D^b(C-{\rm mod})$. Also $D^b(F)=D^b(k-{\rm mod})$.

Perverse sheaves on $X$ which are smooth along a given stratification, are objects of the heart of the  t-structure on $D^b(X)$ glued along the standard t-structures on $D^b(U)$ and on $D^b(F)$, one of which is sheafted by the translation functor ({\em cf.} \cite{GM}). For the purpose to demonstrate the connection of Discrete Harmonic Analysis with perverse sheaves on stratified topological spaces, we restrict our attention in this paper to two examples only of Riemann surfaces (one of which is singular) stratified by a point and the complement to it. In this case we will be interested only in the so-called middle perversity which is defined by the shift by 1 of the t-structure on the closed stratum. In the case of the sphere with a double point, we show that the category of perverse sheaves is equivalent to the category of representations of the homotope constructed via a cyclic graph and a suitable choice of the generalised Laplace operator on it. 

The problem of describing perverse sheaves on various stratified topological spaces currently attracts the growing attention (see \cite{KS1}, \cite{KS2}), because perverse sheaves have turned out to be important for the description of schobers, the categorical constructions which provide us with understanding of the internal structure of various categories of geometric origin   \cite{BKS}, \cite{HK}, \cite{Do}. Schobers are categorifications of perverse sheaves (see  \cite{KS}), whence the clear understanding of the structure of perverse sheaves helps to construct schobers.

\subsection{Perverse sheaves on a disc}
\label{pervdisc}
Let $X$ be an open disc in the complex plane $\mathbb C$ centered at the origin ${\bf 0}$ stratified by the closed subset $F={\bf 0}\in X$ and the complement $U=X\setminus {\bf 0}$. The category of perverse sheaves for this case was scrutinized by A. Beilinson and P. Deligne in \cite{Bei}, \cite{Del}.

The fundamental group of $U$ is ${\mathbb Z}$. Denote by $C=k [x,x^{-1}]$ its group algebra, the algebra of Laurent polynomials in one variable. Its category of (finite dimensional) modules is identified with locally constant sheaves on $U$. The derived category on the stratified disc $X$ has a semiorthogonal decomposition into a pair of subcategories, $D^b(C-{\rm mod})$ and $D^b(k-{\rm mod})$, the latter being the category of sheaves on the point ${\bf 0}$.

Let us calculate the gluing DG-functor $\Phi :D^b(C-{\rm mod})\to D^b(k-{\rm mod})$. The functor is given by a DG $k-C$-bimodule $N$. Since the left action by $k$ coincides with the right one by $k\subset C$, only the right $C$-module structure matters. If $M$ is a left DG $C$-module, then the functor takes it to $\Phi (M)=N\otimes_C^{\mathbb L} M$. Hence, if we apply it to $M=C$, then the result will be $N$ itself. The right $C$ module structure on $N = \Phi (C)$ comes from the right $C$-module structure on $C$.

The gluing functor is a DG-version of the functor that takes a complex of sheaves $M$ on $U$ to the cone of $j_!M\to j_*M$. Thus, we need to calculate in the consistent way functors $j_!$ and $j_*$ for $M$ equal to $k [x,x^{-1}]$, i.e. for the universal local system.

Consider the cover of $U$ by the set of $n$ open subsets $U_k$, where $k$ is taken modulo $n$ and $U_k$ consists of points in $U$ with the phase in the interval 
$$
(\frac{2\pi k}{n}-\epsilon ,\frac{2\pi(k+1)}{n}  +\epsilon ).
$$
It will be convenient for us to assume that $n\ge 2$.  
The intersection $V_k:=U_k\cap U_{k+1}$ 
consists of points with the phase in the interval 
$$
(\frac{2\pi(k+1)}{n}-\epsilon , \frac{2\pi(k+1)}{n}+\epsilon ).
$$ 
Denote by $j_k:U_k\to U$ and $s_k: V_k\to U$ the corresponding open embeddings. Then sheaves $j_{k!}F_k$, where $F_k$ is the constant sheaf on $U_k$, and sheaves $s_{k!}G_k$, where $G_k$ is the constant sheaf on $V_k$, are acyclic for both functors $j_!$ and $j_*$. The restriction of the universal local system to every open set $U_k$ is a constant sheaf. Therefore, we can use the following variant of the Chech resolution induced by our covering for the universal local system ${\cal U}$ for calculating the gluing functor:
\begin{equation}\label{chech-universal}
0\to \oplus_k s_{k!}s_k^*{\cal U}\to \oplus_k j_{k!}j_k^*{\cal U}\to {\cal U}\to 0
\end{equation}

\begin{Theorem}\label{disctheorem} The category of perverse sheaves on the disc $X$ is equivalent to the category of finite dimensional representations of the homotope $B$ constructed by algebra $A={\rm End}(k[x,x^{-1}]^{\oplus n})$
and element  $\Delta\in A$ defined by the operator with matrix
\begin{equation}\label{deltafordisc}
\begin{pmatrix}
1 & -1 & 0 & ... & 0 & 0\\
0 & 1 & -1 & 0 & ... & 0\\
0 & 0 & 1 & -1 &  ... & 0\\
... & ... & ... & ... & ... & ... \\
0 & ... & 0 & 0 & 1 & -1 \\
-x & 0 & ... & 0 & 0& 1
\end{pmatrix}
\end{equation}
\end{Theorem}
\begin{proof}
Apply funtors $j_!$ and $j_*$ termwise to the resolution (\ref{chech-universal}) and take the cone of $j_!\to j_*$. We know that we get a complex of sheaves with cohomology support in  ${\bf 0}$. Thus, we have to look at the fibers of our sheaves only at ${\bf 0}$. Since the fiber over ${\bf 0}$ of the image of $j_!$ of any sheaf is zero, it would be enough to apply functor $j_*$ to (\ref{chech-universal}) and check the fiber over ${\bf 0}$. 

Sheaves $j_k^*{\cal U}$ and $s_k^*{\cal U}$ are constant on the corresponding open sets  with the fibers rank 1 free $k[x,x^{-1}]$-modules. The same are the fibers of  $j_*j_{k!}j_k^*{\cal U}$ and $j_*s_{k!}s_k^*{\cal U}$ over ${\bf 0}$. We can choose the generators in these modules in a consistent way so that the functors of the restriction from $U_k$ to $V_k$, for $k=1,\dots , n$, and from $U_{k+1}$ to $V_k$, for $k=1,\dots n-1$, are the identity operators in $k[x,x^{-1}]$. Then the matrix of the first nontrivial morphism in the sequence (\ref{chech-universal}) after application of $j_*$ in the fiber over ${\bf 0}$ in the chosen bases has the form (\ref{deltafordisc}). 

Similarly to the proof of lemma 18
for the Laplace operator, one can check that that operator (\ref{deltafordisc}) is well-tempered.
The category of perverse sheaves is the heart of the glued t-structure. By theorem 
67, this heart is identified with the category of modules over the homotope constructed via $\Delta$.
\end{proof}


Let us describe the generators and defining relations of the homotope algebra $B$ in the theorem \ref{disctheorem}. We can choose as generators the minimal diagonal projectors of the matrix algebra. Denote their images in the augmentation ideal $B^+$ by $z_i$, where $i$ is better understood modulo $n$. Ideal $B^+$ is a module over algebra $A$, in particular, it is so over $C=k[x,x^{-1}]$. Note that the whole algebra $B$ does not carry any natural $C$-module structure, and since $B^+$ does not have a unit, then it makes no sense to consider $C$ embedded into $B^+$ or in $B$. In particular, $x$ is not an element of algebra $B$, it only acts on elements of $B^+$. The following relations define $B^+$ as the quotient of the free non-unital $C$-algebra with generators $z_i$, $i\in {\mathbb Z}\mod n$:
\begin{itemize}
\item{$z^2_i = z_i$, for any $i$,}

\item{$z_iz_j= 0$, if $j\ne i-1,i,i+1$ modulo $n$,}

\item{$z_iz_{i+1}\dots z_{i+n} = (-1)^nxz_i$, for any $i$.}

\end{itemize}

A more detailed description of this algebra and its generalisations, as well as its representation theory will be given in \cite{BL}. According to theorem \ref{disctheorem},
the algebras described above are Morita equivalent for $n\ge 2$, and their categories of finite dimensional representations are equivalent to the category of perverse sheaves on the stratified disc. The interested reader can find in {\em loc.cit.} or solve as an exercise the translation of our description  into the standard description of perverse sheaves on the disc in terms of nearby and vanishing cycles, which is given by a pair of vector spaces  $U$ and $V$, and a pair of linear operators $a: U\to V$ and $b:V\to U$, satisfying the condition of invertibility of $ab-{\rm id}_V$ and $ba-{\rm id}_U$. 

Generators $z_i$ are idempotents. If we consider the subspaces $E_i:= {\rm Im}z_i$ in an arbitrary representation $E_0$ of algebra $B$, the embedding operators $\delta_i: E_i\to E_0$ and operators $\gamma_i: E_0\to E_i$ induced by $z_i$, then it is easy to see that finite dimensional representations of $B$ are identified with the description of perverse sheaves  on the disc given in \cite{KS2} (with keeping notation as in {\em loc. cit.}).



\subsection{Perverse sheaves on the sphere with a double point and Laplacian of the cyclic graph}


In order to demonstrate how the theory of perverse sheaves is related to the categorical harmonic analysis on graphs, we give the simplest example when the category of perverse sheaves on a stratified space is equivalent the category of finite dimensional representations of $B(\Gamma )$ for a suitable generalised Laplace operator. The graph in this case is cyclic. The category of representations of the homotope for this graph is discussed in section \ref{cyclicgraph}. 

The generalisation of this example to other stratified topological spaces deserves a special attention, it will be considered separately in another publication. 

\begin{Theorem}
Consider the singular Riemann surface $X$, the sphere with one double point $p$, and stratify it with the singular point and the complement to it. Then the category of perverse sheaves on $X$ is equivalent to the category  of finite dimensional representations of the homotope $B(\Gamma )$, where graph  $\Gamma$ is cyclic and parameters $s_{ii+1}$ in (\ref{laplcyc}) are chosen in such a way that the Laplace operator has corank 2 in $x=1$. 

\end{Theorem}
{\bf A sketch of the proof.} 
The complement $X\setminus p$ is homeomorphic to a punctured disc $U$. Let us use its cover that we considered in section \ref{pervdisc} and resolution (\ref{chech-universal}) of the universal local system ${\cal U}$ for calculation of the gluing functor, as in the proof of theorem \ref{deltafordisc}. Since $U$ adjoins the point $p$ along the two brunches of the surface at the point $p$, the same argument as in the proof of the theorem implies that the operator $\Delta :k[x,x^{-1}]^{2n}\to k[x,x^{-1}]^{2n}$ is given by the matrix which is the direct sum of two matrices of the form (\ref{deltafordisc}). 

Matrix (\ref{deltafordisc}) has a one dimensional kernel at $x=1$. This easily implies that the cokernel of $\Delta$ is isomorphic to the direct sum of two $k[x,x^{-1}]$-modules of the form $k[x,x^{-1}]/ (x-1)$. Laplace operator (\ref{laplcyc}) of the cyclic graph satisfying the  conditions as in the formulation of the theorem has as the kernel the same module  $(k[x,x^{-1}]/ (x-1))^{\oplus 2}$. This follows from the computation of the determinant of Laplace operator in section \ref{cyclicgraph}, because this determinant as a function on $x$ has only two zeros. 

It remains to check that the homotopes constructed by two operators of the form  $k[x,x^{-1}]^{m}\to k[x,x^{-1}]^{m}$ and $k[x,x^{-1}]^{n}\to k[x,x^{-1}]^{n}$ with isomorphic cokernels as $k[x,x^{-1}]$-modules, are Morita equivalent. Two such operators differ from each other up to isomorphism by a direct summand, the identity operator $k[x,x^{-1}]^{l}\to k[x,x^{-1}]^{l}$. It is easy to see that adding of this direct summand implies Morita equivalence at the level of homotopes.




We are grateful to Alexander Kuznetsov for fruitful discussions and to Alexander Efimov and Svetlana Makarova for useful remarks to the original version of the text.

\def\cprime{$'$}
\ifx\undefined\bysame
\newcommand{\bysame}{\leavevmode\hbox to3em{\hrulefill}\,}
\fi

{\bf Alexey Bondal}

Steklov Math Institute of RAS;

MIPT, AGHA laboratory;

National Research University Higher School of Economics;

Kavli IPMU, Tokyo University.

{\bf Ilya Zhdanovskiy} 

MIPT, AGHA laboratory;

National Research University Higher School of Economics.

\end{document}